\documentclass[]{interact}

\usepackage{epstopdf}% To incorporate .eps illustrations using PDFLaTeX, etc.
\usepackage[caption=false]{subfig}
\usepackage{enumitem}
\usepackage[ruled,lined,boxruled,titlenotnumbered]{algorithm2e}

\usepackage[numbers,sort&compress]{natbib}% Citation support using natbib.sty
\bibpunct[, ]{[}{]}{,}{n}{,}{,}% Citation support using natbib.sty
\makeatletter% @ becomes a letter
\def\NAT@def@citea{\def\@citea{\NAT@separator}}% Suppress spaces between citations using natbib.sty
\makeatother% @ becomes a symbol again

\theoremstyle{plain}% Theorem-like structures provided by amsthm.sty
\newtheorem{theorem}{Theorem}[section]
\newtheorem{lemma}[theorem]{Lemma}
\newtheorem{corollary}[theorem]{Corollary}

\theoremstyle{definition}

\theoremstyle{remark}
\newtheorem{remark}{Remark}

\usepackage[colorlinks=true]{hyperref}

\newcommand{\R}{{\mathbb R}}
\newcommand{\N}{{\mathbb N}}
\DeclareMathOperator{\argmin}{argmin}
\DeclareMathOperator{\prox}{prox}

\newcommand{\cC}{{\mathcal C}}
\newcommand{\cE}{{\mathcal E}}
\newcommand{\cG}{{\mathcal G}}
\newcommand{\cH}{{\mathcal H}}
\newcommand{\cO}{{\mathcal O}}

\newcommand{\demi}{\frac{1}{2}}
\newcommand{\ie}{{\it i.e.}\,\,}

\newcommand{\Rb}{\R\cup\{+\infty\}}

\newcommand{\eqdef}{:=}
\newcommand{\seq}[1]{\pa{#1}_{k \in \N}}

\newcommand{\dotp}[2]{\langle #1,\,#2 \rangle}
\newcommand{\norm}[1]{\left\|{#1}\right\|}
\newcommand{\pa}[1]{\left({#1}\right)}
\newcommand{\bpa}[1]{\Big({#1}\Big)}
\newcommand{\brac}[1]{\left[{#1}\right]}
\newcommand{\DS}[1]{\displaystyle{#1}}
\newcommand{\IPAHD}{\mathrm{(IPAHD)}}
\newcommand{\IPAHDNS}{\text{\rm{(IPAHD-NS)}\,}}

\newcommand{\IGAHD}{\text{\rm{(IGAHD)}\,}}

\newcommand{\AVD}[1]{\text{\rm{(AVD)}$_{#1}$\,}}
\newcommand{\DINAVD}[1]{\text{\rm{(DIN-AVD)}$_{#1}$\,}}

\newcommand{\xkp}{x_{k+1}}

\newcommand{\qandq}{\quad \text{and} \quad}

\usepackage{ulem}

\if
{
  
 \usepackage[pageref]{backref}
\renewcommand*{\backrefalt}[4]{%
\ifcase #1 %
(Not cited)%
\or
(Cited on p.~#2)%
\else
(Cited on pp.~#2)%
\fi
}

}
\fi

\begin{document}
\title{Convergence of iterates for first-order optimization algorithms with inertia and Hessian driven damping}
\author{
\name{Hedy Attouch\textsuperscript{a}\thanks{CONTACT J. Fadili. Email: Jalal.Fadili@greyc.ensicaen.fr}, 
Zaki Chbani\textsuperscript{b}, Jalal Fadili\textsuperscript{c} and Hassan Riahi\textsuperscript{d}}
\affil{\textsuperscript{a}IMAG, Univ. Montpellier, CNRS, Montpellier, France;\\ \textsuperscript{b}\textsuperscript{d}Cadi Ayyad Univ., Faculty of Sciences S\'emlalia, Mathematics, 40000 Marrakech, Morocco;\\ \textsuperscript{c}Normandie Univ, ENSICAEN, CNRS, GREYC, Caen, France.}
}

\maketitle

\begin{abstract}
In a Hilbert space setting, for convex optimization, we show the convergence of the iterates to optimal solutions for a class of accelerated first-order algorithms. They can be interpreted as discrete temporal versions of an inertial dynamic involving both viscous damping and Hessian-driven damping. The asymptotically vanishing viscous damping is linked to the accelerated gradient method of Nesterov while the Hessian driven damping makes it possible to significantly attenuate the oscillations. By treating the Hessian-driven damping as the time derivative of the gradient term, this gives, in discretized form, first-order algorithms. These results complement the previous work of the authors where it was shown the fast convergence of the values, and the fast convergence towards zero of the gradients.
\end{abstract}


%\tableofcontents

%%%%%%%%%%%%%%%%%%%%%%%%%%%%%%%%%%%%%%%%%%%%%%%%%%%%%%%%%%%%%%%%%%%%%%%%%%%%%%%%%%%
\section{Introduction}\label{sec:prel} 
%%%%%%%%%%%%%%%%%%%%%%%%%%%%%%%%%%%%%%%%%%%%%%%%%%%%%%%%%%%%%%%%%%%%%%%%%%%%%%%%%%%
%
Unless specified, throughout the paper we make the following assumptions \footnote{In fact, all the algorithmic results require only $f$ to be convex differentiable; it is only when we consider the second order evolution system that we need the second-order derivatives of $f$.}
\[
\boxed{ 
\begin{cases}
\cH \text{ is a real Hilbert space};\\
f: \cH \rightarrow \R  \text{ is a convex function of  class }  \cC^2, \,  S \eqdef \argmin_{\cH} f \neq \emptyset;  \\
\gamma, \,  \beta,  \, b: [t_0, +\infty[ \to \R^+ \text{ are non-negative continuous functions}, \,  t_0 >0.
\end{cases}
}
\]
{Our first objective is to study the convergence to optimal solutions}, when $t \to +\infty$, of the trajectories of  the inertial system with Hessian-driven damping
\begin{equation}\label{basic_eq}
\qquad \ddot{x}(t) + \gamma(t)\dot{x}(t) +  \beta (t) \nabla^2  f (x(t)) \dot{x} (t) + b(t)\nabla  f (x(t)) = 0.
\end{equation} 
In \eqref{basic_eq}, $\gamma(t)$ is the viscous damping parameter, $\beta(t)$ is the Hessian-driven damping parameter, and $b(t)$ is a time scale parameter.
{Then, we will study the convergence of the iterates for the first order optimization algorithms obtained by
temporal discretization of this system.} 
At first glance, the presence of the Hessian may seem to entail numerical difficulties. However, this is not the case as the Hessian $\nabla^2  f $ of $f$ intervenes in the above ODE in the form $\nabla^2  f (x(t)) \dot{x} (t)$, which is nothing but the derivative with respect to time of the mapping $t \mapsto \nabla  f (x(t))$. As a consequence, finite-difference time discretization of this dynamic provides first-order algorithms of the form
\[
\begin{cases}
y_k = x_{k} + \alpha_k ( x_{k}  - x_{k-1}) - \beta_k  \left( \nabla f (x_k)  - \nabla f (x_{k-1}) \right) \\
x_{k+1} = T (y_k) ,
\end{cases}
\]
where the Nesterov extrapolation scheme (\cite{Nest1,Nest2}) is modified by the introduction of the difference of the gradients at consecutive iterates.
The operator $T$, to be specified later, will be directly linked to the gradient of $f$, or to the proximal operator of $f$.
While retaining the fast convergence of the function values reminiscent of the Nesterov
accelerated algorithm, it is shown in \cite{ACFR} that these algorithms enjoy additional favorable properties among which the most important are:
\begin{enumerate}[label=$\bullet$]
\item fast convergence of the gradient towards zero;
\item attenuation of the oscillations;
\item extension to the non-smooth setting;
\item acceleration via the time scaling factor.
\end{enumerate}

\paragraph*{Contributions and relation to prior work} 
In \cite{ACFR}, global convergence of the trajectories (for the continuous dynamic) and of the iterates (for the corresponding algorithms) to optimal solutions remained an open issue. It is our chief goal in this paper to fill in this gap by answering these questions. We also establish that under certain choices of the parameters, the fast convergence rates obtained in \cite{ACFR} can be improved from $\cO(\cdot)$ to ${\rm o}(\cdot)$. By complementing the work of \cite{ACFR}, our new results provide a deep understanding of the behaviour of the system \eqref{basic_eq} and the corresponding discrete algorithms.  

Due to the remarkable properties induced by the presence of the Hessian-driven damping, these inertial dynamics and algorithms have been the subject of active recent developments; see  
the work of \cite{SDJS} on a high resolution perspective, that of
 \cite{BCPF} for application to deep learning, the papers of \cite{AAt1,AAt2} for combination with dry friction, those of \cite{AL1,AL2} and \cite{Kim} for the case of monotone inclusions, of \cite{Kim-F} for optimization algorithms, the work of \cite{ABCR} for handling temporal scaling, the control perspective promoted in \cite{JL}, the damping as a closed-loop control studied in \cite{ABC}, and \cite{BCL} where combination with Tikhonov regularization is studied.
%%%%%%%%

% We will first study the continuous dynamic which will serve us as a guide for the study of the associated algorithms.

The strategy underlying our proof is built upon Lyapunov analysis with properly designed Lyapunov functions. For the convergence of  values, it suffices to use a carefully chosen energy-type Lyapunov function. By contrast, to show the convergence of the iterates and trajectories, the proof is more involved  and requires the use of a whole family of Lyapunov functions. The convergence of these functions when time tends to infinity then leads to the convergence of their differences. This gives the convergence of the anchor functions, and this is precisely what makes it possible to conclude. It is worth noting that this strategy is already present in the proof of the convergence of the greatest slope for convex functions by Bruck \cite{Bruck}, thanks to Opial's lemma \cite{Opial} (see Lemma~\ref{le.Opial}).

\paragraph*{Contents} The paper is organized as follows. In section~\ref{sec:continuous}, we complete the study carried out in \cite{ACFR} of the continuous dynamic, obtain  additional estimations, and prove the weak convergence of the trajectories.  This will serve as a guide for the study of the convergence of the associated algorithms for which we obtain parallel results. In section~\ref{sec:prox} and we establish the convergence of the iterates for the corresponding proximal algorithms, then in section~\ref{sec:gradient} we consider the gradient algorithms. In section~\ref{sec:numerics} we illustrate this study with an application to Lasso-type problems.

%%%%%%%%%%%%%%%%%%%%%%%%%%%%%%%%%%%%%%%%%%%%%%%%%%%%%%%%%%%%%%%%%%%%%%%%%%%%%%%%%%%
\section{Continuous dynamic: convergence of trajectories}\label{sec:continuous}
%%%%%%%%%%%%%%%%%%%%%%%%%%%%%%%%%%%%%%%%%%%%%%%%%%%%%%%%%%%%%%%%%%%%%%%%%%%%%%%%%%%
%
As a preparatory step to the study of the convergence of the associated algorithms, obtained by temporal discretization, we start by analyzing the convergence properties, as $t \to +\infty$, of the trajectories generated by the dynamic 
\[
\boxed{ 
{\DINAVD{\alpha, \beta, b}}\quad  
\ddot{x}(t) + \frac{\alpha}{t} \dot{x}(t) +  \beta (t) \nabla^2  f (x(t)) \dot{x} (t) + b(t)\nabla  f (x(t)) = 0.
}
\]
The Hessian-driven damping is related to the Dynamic Inertial Newton method, and the viscous damping coefficient $\gamma(t)=\frac{\alpha}{t} $ vanishes as $t\to +\infty$ (Asymptotic Vanishing Damping), hence the terminology.
We limit our study to this choice of the viscous damping coefficient where $\alpha >0$,  because it is the most interesting case. Indeed, it is closely related to the accelerated gradient method of Nesterov, and provides an optimal convergence rate of the values, as we will recall shortly. We consider however a general coefficient $b(t)$, which allows us to take advantage of the temporal scaling aspects, and is useful for applications.

%%%%%%%%%%%%%%%%%%%%%%%%%%%%%%
\subsection{Historical overview}
Before delving into the analysis of ${\DINAVD{\alpha, \beta, b}}$, let us review the main convergence properties known for special cases of it.

\paragraph*{Case $\beta \equiv 0$}
When the Hessian-driven damping is dropped and $b(t)\equiv 1$, the system specializes to
\begin{equation*}
\AVD{\alpha} \qquad 
\ddot{x}(t) + \frac{\alpha}{t} \dot{x}(t)  +  \nabla f (x(t)) = 0.
\end{equation*}
It was introduced in the context of convex optimization in \cite{SBC}. For a general convex differentiable function $f$, it provides a continuous version of the accelerated gradient method of Nesterov \cite{Nest1,Nest2}. For $\alpha \geq 3$, each trajectory $x(\cdot)$ of $\AVD{\alpha}$  satisfies the asymptotic  convergence rate of the values $f(x(t)) - \min_{\cH}f =\cO \left(1 /t^2\right)$. For $\alpha >3$, it was shown in \cite{ACPR} that each trajectory converges weakly to a minimizer of $f$. For such values of $\alpha$, it was also proved in \cite{AP} and \cite{May} that the asymptotic convergence rate of the values is actually ${\rm o}(1/t^2)$.
%As a specific feature, the viscous damping coefficient $\frac{\alpha}{t}$ vanishes (tends to zero) as time $t$ goes to infinity, hence the terminology.
The case  $\alpha = 3$, which corresponds to Nesterov's historical algorithm, is critical. In particular, for this critical value, the question of the convergence of the trajectories remains an open problem (except in one dimension where convergence holds \cite{ACR-subcrit}).
The subcritical case $\alpha \leq 3$ has  been examined in \cite{AAD} and \cite{ACR-subcrit}, with the convergence rate of the objective values $\cO\left( t^{-\frac{2\alpha}{3}}\right)$.
These rates are optimal, \ie they can be reached, or approached arbitrarily close.
\if
{
The corresponding inertial algorithms
\begin{eqnarray*}
\begin{array}{rcl}
\begin{cases}
y_k=  \displaystyle{x_{k} + \left(1 - \frac{\alpha}{k} \right)( x_{k}  - x_{k-1})} \hspace{2.1cm} \vspace{3mm} \\
x_{k+1} = y_k - s \nabla f(y_k)
\end{cases}
\vspace{2mm}
\end{array}
\end{eqnarray*}
are in line with the Nesterov accelerated gradient method. They enjoy similar properties to the continuous case, see Chambolle-Dossal \cite{CD}, and \cite{AC2}, \cite{ACPR}, \cite{AP} for further results.
}
\fi

\paragraph*{Case $\beta >0$}
The Hessian driven damping was first introduced in \cite{AABR}, \cite{AMR} when combined with a viscous damping term whose coefficient is fixed. In particular, the inertial system
\[
\DINAVD{\alpha, \beta} \qquad 
\ddot{x}(t) + \frac{\alpha}{t}\dot{x}(t) +  \beta \nabla^2  f (x(t)) \dot{x} (t) + \nabla  f (x(t)) = 0
\]
was introduced in \cite{APR2}. It combines the asymptotic vanishing  damping (associated with the Nesterov method) with the Hessian-driven damping. At a first glance, this system looks more complicated than \AVD{\alpha}. In fact, in \cite{APR2}, it is shown that
$\DINAVD{\alpha, \beta}$ is equivalent to the first-order system in time and space
\[
\begin{cases}
\dot{x}(t) +  \beta \nabla f (x(t))- \pa{\frac{1}{\beta} -  \frac{\alpha}{t}}  x(t) +  \frac{1}{\beta} y(t) = 0;  	 \\
\dot{y}(t) -  \pa{\frac{1}{\beta} -  \frac{\alpha}{t} +  \frac{\alpha \beta}{t^2}} x(t)
 + \frac{1}{\beta} y(t) =0.
\end{cases}
\]
This provides a natural extension to the case where $f: \cH \to \Rb$ is a proper lower semicontinuous (lsc) and convex function, just replacing the gradient operator $\nabla f$ by the subdifferential $\partial f$. While preserving the convergence properties of the Nesterov  accelerated method, $\DINAVD{\alpha, \beta}$ provides fast convergence to zero of the gradients, and has a taming effect on the oscillations. {More precisely, when $\alpha > 3$, one has
\[
f(x(t))-  \min_{\cH}f = {\rm o}\pa{1/t^2} \; \mbox{ as} \;  t\to +\infty \qandq \int_{t_0}^{+\infty} t^2 \|\nabla f (x(t)) \|^2  dt    < + \infty.
\]
}
The extension of these results to the case of general parameters $\gamma(t)$, $\beta(t)$ and $b(t)$ has been obtained in \cite{ABCR} and \cite{ACFR}.

%%%%%%%%%%%%%%%%%%%%%%%%%%%%%%
\subsection{Convergence rates}
Let us first bring some complements concerning the rate of convergence of the values obtained in \cite{ACFR}. They will be very useful in the following section devoted to the convergence of trajectories. Observe that by assuming $t_0 >0$, we circumvent the difficulties raised by the singularity of the damping coefficient $\frac{\alpha}{t}$ at the origin. This is however by no means restrictive in our setting since we are primarily interested in asymptotic analysis.

\noindent 
To lighten notation, we introduce the following function $w: [t_0, +\infty[ \to \R$ which plays a key role in our analysis:
\begin{equation}\label{def-w}
w(t) \eqdef b(t)-\dot{\beta}(t) -\dfrac{\beta(t)}{t}.
\end{equation}
\begin{theorem}\label{ACFR,rescale}
Take $\alpha \geq 1$. Let  $x: [t_0, +\infty[ \rightarrow \cH$ be a solution trajectory of  $ \DINAVD{\alpha, \beta, b}$. Suppose that the following conditions are satisfied: for all $t\geq t_0$
\begin{enumerate}[label={\rm($\cC_{\arabic*}$)},itemindent=10ex]
\item $b(t) > \dot{\beta}(t) +\dfrac{\beta(t)}{t}$; \label{cond:C1}
\item $(\alpha-3)w(t) - t\dot{w}(t) \geq  0$. \label{cond:C2}
\end{enumerate}
Then, $w(t)$ is positive and there exists a positive constant $C_0$ such that
\begin{enumerate}[label={\rm(\roman*)}]
\item \label{ACFR,rescale:claim1}
$0\leq f(x(t))-\min_{\cH} f \leq \dfrac{C_0}{t^2 w(t)} \, \mbox{ for all } \, t \geq t_0$;
\item \label{ACFR,rescale:claim2}
$\DS{\int_{t_{0}}^{+\infty} t^2 \beta(t) w(t)\norm{\nabla f(x(t))}^{2} dt<+\infty}$;
\item \label{ACFR,rescale:claim3}
$\DS{\int_{t_{0}}^{+\infty} t \pa{(\alpha-3)w(t) - t\dot{w}(t)}(f(x(t))-\min_{\cH}f) dt <+\infty}$.
\end{enumerate}
%%%%%%%%%%%%%%%%
Suppose moreover  that $\alpha > 1$, and that for all $ t \geq t_0$,
\begin{enumerate}[label={\rm($\cC_{\arabic*}$)},itemindent=10ex,start=3]
\item $(\alpha-3)w(t) - t\dot{w}(t)  \geq   \varepsilon  b(t)$, for some $\varepsilon \in ]0,\alpha-1[$. \label{cond:C3}
\end{enumerate}
Then, 
\begin{enumerate}[label={\rm(\roman*)},start=4]
\item \label{ACFR,rescale:claim4}
$\DS{\int_{t_0}^{+\infty} t b(t) (f(x(t))-\min_{\cH}f) dt <+\infty}$.
\item \label{ACFR,rescale:claim5}
$\DS{\int_{t_0}^{+\infty} t \|\dot x(t)\|^2dt <+\infty}$.
\item \label{ACFR,rescale:claim6}
$\DS{\sup_{t\geq t_0} \|x(t)\|<+\infty}$.
\end{enumerate}
In addition,  
\begin{enumerate}[label={\rm(\roman*)},start=7]
\item \label{ACFR,rescale:claim7}
if $\beta(\cdot)$ is non-decreasing, then $\DS{\int_{t_0}^{+\infty}  t w(t)  \dotp{\nabla f(x(t))}{x(t)-x^\star}  dt <+\infty}$;
\item \label{ACFR,rescale:claim8}
if there exists $C_1 >0$ such that $\DS{\frac{d}{dt} \left(t^2b(t)\right)\leq C_1tb(t)}$ for $t$ large enough, then, as $t \to +\infty$,
\[
f (x(t))-\min_{\cH} f ={\rm o} \pa{\dfrac{1}{ t^2b(t)}} \qandq \|\dot x(t)\| = {\rm o} \pa{\dfrac{1}{t}}.
\]
\end{enumerate}
\end{theorem}
\begin{proof}
Take $x^\star \in S \eqdef \argmin_{\cH} f$. Let us define for $t\geq t_0$
\begin{equation}
E_{\lambda}(t) \eqdef \delta(t)(f(x(t))-f(x^\star))+ \frac{1}{2}\norm{v_{\lambda}(t)}^{2}
+\frac{c}{2}\norm{x(t) -x^\star}^{2}, \label{def:E2}
\end{equation}
where 
\[
v_{\lambda}(t)\eqdef\lambda(x(t)-x^\star)+t\pa{\dot{x}(t)+\beta(t)\nabla f(x(t))}
\]
and $\delta(\cdot)$, $\lambda$, $c$ are positive parameters that will be adjusted along the proof. 
The function $ E_{\lambda}(\cdot)$ will serve as Lyapunov's function.\\
\noindent 
Differentiating $E_{\lambda}$ gives
\begin{eqnarray}
\dfrac{d}{dt} E_{\lambda}(t)&=&\dot{\delta}(t)(f(x(t))-f(x^\star))+\delta(t) \dotp{\nabla f(x(t))}{\dot{x}(t)}+ \dotp{v_{\lambda}(t)}{\dot{v}_{\lambda}(t)} \nonumber \\
 &+& c\dotp{x(t) - x^\star}{\dot{x}(t)}. \label{der-E}
\end{eqnarray}
Using the constitutive equation $ \DINAVD{\alpha, \beta, b}$, we have
\begin{align*}
\dot{v}_{\lambda}(t) 
& = (\lambda +1)\dot{x}(t)+\beta(t) \nabla f(x(t))
+ t\pa{\ddot{x}(t)+\dot{\beta}(t)\nabla f(x(t))+\beta(t) \nabla^{2}f(x(t))\dot{x}(t)}\\ 
&= (\lambda +1)\dot{x}(t)+ \beta(t) \nabla f(x(t))+t\pa{-\frac{\alpha}{t}\dot{x}(t)+(\dot{\beta}(t)-b(t))\nabla f(x(t))}\\ 
&= (\lambda +1 -\alpha)\dot{x}(t) - tw(t)\nabla f(x(t)).
\end{align*}
Therefore,
\begin{align*}
\dotp{v_{\lambda}(t)}{\dot{v}_{\lambda}(t)} 
&=(\lambda +1 -\alpha) \dotp{\lambda(x(t)-x^\star)+t\pa{\dot{x}(t)+\beta(t)\nabla f(x(t))}}{\dot{x}(t)} \\
&- tw(t) \dotp{\lambda(x(t)-x^\star)+t\pa{\dot{x}(t)+\beta(t)\nabla f(x(t))}}{\nabla f(x(t))} \\
&=\lambda (\lambda +1 -\alpha)\dotp{x(t)-x^\star}{\dot{x}(t) } + t (\lambda +1 -\alpha) \|  \dot{x}(t)\|^2 \\
&+ \Big( t \beta(t)(\lambda +1 -\alpha) - t^2 w(t)] \Big)  \dotp{\nabla f(x(t))}{\dot{x}(t)}\\
&- \lambda t w(t) \dotp{\nabla f(x(t))}{x(t)-x^\star } \\
&- t^2 \beta (t) w(t) \|\nabla f(x(t))\|^2 .
\end{align*}
Let us go back to \eqref{der-E}. Take $\delta(t)$ so that the terms $\dotp{\nabla f(x(t))}{\dot{x}(t)}$ cancel. This gives
\begin{align}
\delta(t) 
&= t^2 \Big(b(t) -\dot{\beta}(t)-\dfrac{\beta(t)}{t}\Big) -t \beta (t) (\lambda +1 -\alpha) \nonumber \\
&= t^2 w(t)- (\lambda +1 -\alpha)t \beta (t). \label{def:delta}
\end{align}
Also take $c$ so that the terms $\dotp{x(t)-x^\star}{\dot{x}(t) }$ cancel. This gives
\begin{equation}\label{def:c}
c = -\lambda (\lambda +1 -\alpha).
\end{equation}
Since we want $c$ to be non-negative, we take $\lambda > 0$ such that
\begin{equation}\label{lambda_bound}
\lambda \leq \alpha -1.
\end{equation}
From \eqref{lambda_bound}, $w(t) >0$ (by assumption \ref{cond:C1}), and the definition \eqref{def:delta} of $\delta(\cdot)$, we have that $\delta (\cdot)$ is non-negative.
Using this into \eqref{der-E}, the latter reads
\begin{align}
\dfrac{d}{dt}E_{\lambda}(t)
&= \dot{\delta}(t)(f(x(t))-f(x^\star))-\lambda t w(t) \dotp{\nabla f(x(t))}{x(t)-x^\star } \nonumber\\
&- t^2 \beta (t) w(t) \|\nabla f(x(t))\|^2 +    t(\lambda +1 -\alpha)\|  \dot{x}(t)\|^2 \nonumber\\
&= \dot{\delta}(t)(f(x(t))-f(x^\star)) - \lambda t w(t) \dotp{\nabla f(x(t))}{x(t)-x^\star } \nonumber\\
&- t^2 \beta (t) w(t)       \|\nabla f(x(t))\|^2 +    t(\lambda +1 -\alpha)\|  \dot{x}(t)\|^2 \label{integral_est},
\end{align}
where we used definition \eqref{def-w} of $w(t)$.
By convexity of $f$, we have
\[
f(x^\star)-f(x(t)) \geq \dotp{\nabla f(x(t))}{x^\star-x(t)} ,
\]
and thus \eqref{integral_est} becomes
\begin{multline}
\dfrac{d}{dt}E_{\lambda}(t)+t^2\beta(t)w(t)\norm{\nabla f(x(t))}^{2}+  \brac{\lambda t w(t) -\dot{\delta}(t)}(f(x(t))-f(x^\star)) \\
\leq  t(\lambda + 1 -\alpha)\|  \dot{x}(t)\|^2 . \label{bacic-Liap-22}
\end{multline}
In view of \eqref{lambda_bound}, we obtain
\begin{equation}\label{bacic-Liap-22-a}
\dfrac{d}{dt}E_{\lambda}(t)+t^2\beta(t)w(t)\norm{\nabla f(x(t))}^{2}+\brac{\lambda t w(t) -\dot{\delta}(t)}(f(x(t))-f(x^\star)) \leq  0 . 
\end{equation}
To go further, we must take into account the sign of
\begin{equation}\label{bacic-Liap-22-b}
\lambda t w(t) -\dot{\delta}(t)= t \pa{(\lambda-2)w(t) - t\dot{w}(t)} + (\lambda +1 -\alpha) (\beta (t) + t \dot{\beta}(t)),
\end{equation}
which  depends on the  value of the parameter $\lambda$. Recall that we are free to choose $\lambda$ under the sole condition that it satisfies \eqref{lambda_bound}.

\paragraph*{Choice $\lambda = \alpha-1$.} This corresponds to the largest possible value for $\lambda$, which is the situation considered in \cite{ACFR}. Then 
\begin{equation}\label{bacic-Liap-22-c}
\lambda t w(t) -\dot{\delta}(t)= t \Big(  (\alpha-3)w(t) - t\dot{w}(t)\Big) .
\end{equation}
According to the hypothesis \ref{cond:C2}, the left hand side in \eqref{bacic-Liap-22-c} is non-negative. Thus \eqref{bacic-Liap-22-a} entails that  $\dfrac{d}{dt}E_{\alpha-1}(t)\leq 0$. In turn, $E_{\alpha-1}(\cdot)$ is non-increasing, and therefore $E_{\alpha-1}(t) \leq E_{\alpha-1}(t_0)$ for all $t\geq t_0$. On the other hand, by \eqref{def:delta}, we have
\[
\delta(t) = t^2 w(t)
\]
which is again non-negative by \ref{cond:C1}. So, $E_{\alpha-1}(t)$ writes (note that, with this choice of $\lambda$, and by \eqref{def:c} we have $c=-\lambda (\lambda +1 -\alpha)=0 $)
\begin{equation*}
E_{\alpha-1}(t)\eqdef t^2 w(t)(f(x(t))-f(x^\star))+ \frac{1}{2}\norm{(\alpha-1)(x(t)-x^\star)+t\bpa{\dot{x}(t)+\beta(t)\nabla f(x(t))}}^{2}. 
%\label{E_decrease_1}
\end{equation*}
Since all the terms that enter $E_{\alpha-1}(\cdot)$ are non-negative, we obtain, for all $t\geq t_0$
\[
f(x(t))-f(x^\star)\leq \dfrac{E_{\alpha-1}(t_0)}{t^{2}w(t)} 
\]
which is claim~\ref{ACFR,rescale:claim1} with $C_0= E_{\alpha -1}(t_0)$.
In addition, by integrating \eqref{bacic-Liap-22-a} we obtain
\[
\int_{t_{0}}^{+\infty} t^2\beta(t)w(t)\norm{\nabla f(x(t))}^{2} dt \leq  E_{\alpha-1}(t_0)  <+\infty ,
\]
and
\[
\int_{t_{0}}^{+\infty} t \Big(  (\alpha-3)w(t) - t\dot{w}(t)\Big)(f(x(t))-f(x^\star)) dt \leq E_{\alpha -1}(t_0)<+\infty ,
\]
which gives statements \ref{ACFR,rescale:claim2} and \ref{ACFR,rescale:claim3}.

\paragraph*{Choice $\lambda = \alpha - 1 - \varepsilon$}, where $\varepsilon > 0$ is given by condition \ref{cond:C3}, which is now supposed to be satisfied.
Then, condition \eqref{lambda_bound} is obviously satisfied, and the the above calculations are still valid until \eqref{bacic-Liap-22} which now reads
\begin{equation}\label{bacic-Liap-33}
\dfrac{d}{dt}E_{\lambda}(t)+t^2\beta(t)w(t)\Vert\nabla f(x(t))\Vert^{2}+  \Big[ \lambda t w(t) -\dot{\delta}(t)\Big](f(x(t))-f(x^\star))\ + \varepsilon  t\|  \dot{x}(t)\|^2 \leq 0. 
\end{equation}
This choice of $\lambda$ and \eqref{def:delta} yield
$
\delta (t) =  t^2 w(t) +\varepsilon t \beta (t).
$ 
Therefore
\begin{eqnarray*}
\lambda t w(t) -\dot{\delta}(t)
&=&  \lambda t w(t) - 2t w(t) - t^2 \dot{w}(t) - \varepsilon  \beta (t) - \varepsilon t \dot{\beta}(t)\\
&=&  (\lambda -2) t w(t) - t^2 \dot{w}(t) - \varepsilon  \beta (t) - \varepsilon t \dot{\beta}(t)\\ 
&=&  (\alpha -3 -\varepsilon) t w(t) - t^2 \dot{w}(t) - \varepsilon t \left(  \dot{\beta}(t)+ \frac{\beta (t)}{t} \right). 
\end{eqnarray*}
Plugging $w(t)$ defined in \eqref{def-w} in the last identity, we get
\begin{eqnarray*}
\lambda t w(t) -\dot{\delta}(t)
&=& (\alpha -3 -\varepsilon) t w(t) - t^2 \dot{w}(t) - \varepsilon t \pa{b(t)- w(t)}\\
%&=& (\alpha -3) t w(t) - t^2 \dot{w}(t) - \varepsilon t   b(t)\\
&=&  t\bpa{(\alpha -3)  w(t) - t \dot{w}(t) - \varepsilon b(t)}.
\end{eqnarray*}
Therefore, $\lambda t w(t) -\dot{\delta}(t)$ is non-negative under the condition \ref{cond:C3}.
By integrating \eqref{bacic-Liap-33}, and since $\varepsilon >0$ we obtain
\[
\int_{t_0}^{+\infty} t \|\dot x(t)\|^2dt <+\infty ,
\]
hence establishing \ref{ACFR,rescale:claim5}.
Since by \eqref{def:c}, $c=\varepsilon (\alpha-1-\varepsilon) >0 $ (recall  $0 <\varepsilon < \alpha -1$), we have that $E_\lambda(t)$ is non-negative, and from \eqref{bacic-Liap-33} that it is is a non-increasing function. In turn, it is bounded from above, and so is $\norm{x(t) -x^\star}^{2}$. Therefore, the trajectory $x(\cdot)$ fulfills claim~\ref{ACFR,rescale:claim6}.

Combining item \ref{ACFR,rescale:claim3} and condition \ref{cond:C3}, we have 
\[
\int_{t_0}^{+\infty} \varepsilon t b(t) (f(x(t))-f(x^\star)) dt \leq
\int_{t_{0}}^{+\infty} t \Big(  (\alpha-3)w(t) - t\dot{w}(t)\Big)(f(x(t))-f(x^\star)) dt <+\infty ,
\]
which is statement~\ref{ACFR,rescale:claim4}.

%\textbf{3. Let's  prove the  estimate $(vii)$.} 
%Take $\lambda = \alpha -1$.
We now turn to showing \ref{ACFR,rescale:claim7}. For this, let $\rho \in ]0,1[$ a positive parameter to be adjusted. We embark from \eqref{integral_est}, and split the term  $\lambda t w(t)  \dotp{\nabla f(x(t))}{x(t)-x^\star }$ into the sum of two terms with respective weights $\rho$, and $1 - \rho$. We then apply  the convex subdifferential inequality to the one with weight $1-\rho$. Doing so, we obtain
\begin{equation}
\dfrac{d}{dt}E_{\lambda}(t) + \Big( (1-\rho) \lambda t w(t)  - \dot{\delta}(t)\Big)(f(x(t))-f(x^\star))+\rho \lambda t w(t)  \dotp{\nabla f(x(t))}{x(t)-x^\star} \leq 0 . \label{integral_est_bb}
\end{equation}
The point is to show that by appropriately choosing $\lambda$ and $\rho$, we can make the quantity
\[
A(t)\eqdef (1-\rho) \lambda t w(t)  - \dot{\delta}(t)
\]
non-negative. Take  $\lambda = \alpha -1$. Hence 
$
\delta (t) =  t^2 w(t) .
$ 
The same calculation as above gives
\begin{align*}
A(t)
&= t\bpa{(1-\rho)(\alpha -1) w(t) - 2 w(t) - t \dot{w}(t)} \\
&= t\bpa{\pa{(\alpha -3) -\rho (\alpha -1)} w(t) - t \dot{w}(t)}.
\end{align*}
Take $\rho= \DS{\frac{\varepsilon}{\alpha-1}} \in ]0,1[$, where $\varepsilon$ is given by condition \ref{cond:C3}. Then, 
\begin{equation}\label{def:A}
A(t)= t\bpa{(\alpha -3 -\varepsilon ) w(t) - t \dot{w}(t)}.
\end{equation}
By definition of $w$, and since $\beta$ has been supposed non-decreasing, we have
\begin{equation}\label{def-w-consequence}
b(t) = w(t)+ \dot{\beta}(t) +\dfrac{\beta(t)}{t} \geq w(t).
\end{equation}
Consequently, \ref{cond:C3} entails $(\alpha-3)w(t) - t\dot{w}(t)  \geq   \varepsilon  w(t)$, and thus the quantity $A(t)$ in \eqref{def:A} is non-negative, {as desired}. Integrating \eqref{integral_est_bb}, and since $\lambda = \alpha - 1 > 0$, $\rho>0$, and $E_\lambda$ is bounded from below (in fact non-negative), we obtain claim \ref{ACFR,rescale:claim7}.
%$$ \, \, \, \,
%\int_{t_0}^{+\infty}  t w(t)  \dotp{\nabla f(x(t))}{x(t)-x^\star }  dt <+\infty, 
%$$

It remains to prove \ref{ACFR,rescale:claim8}. {Taking}  the inner product of $\DINAVD{\alpha, \beta, b}$ with $t^2\dot x(t)$, we obtain using the chain rule and non-negative definiteness of $\nabla^2 f(x(t))$,
\begin{eqnarray*}
0&=&  t^2 \dotp{\ddot x(t)}{\dot x(t)} +\alpha t\|\dot x(t)\|^2 + t^2 \beta(t) \dotp{\nabla^2 f (x(t))\dot x(t)}{\dot x(t)}  + t^2b(t) \dotp{\nabla f (x(t))}{\dot x(t)} \\
&\geq& t^2 \frac{d}{dt}\pa{\frac12 \| \dot x(t) \|^2} + \alpha t \|\dot x(t)\|^2   + t^2b(t) \frac{d}{dt}\pa{ f (x(t))-\min_{\cH} f}\\
&=& \frac{d}{dt}\pa{\frac{ t^2}2 \| \dot x(t) \|^2 + t^2b(t) \pa{f (x(t))-\min_{\cH} f}} + (\alpha-1) t \|\dot x(t)\|^2   \\
&&- \pa{f (x(t))-\min_{\cH} f}\frac{d}{dt} \pa{t^2b(t)}.
\end{eqnarray*}
Integrating from $s$ to $t$, we get
\begin{eqnarray*}
0&\geq &  \frac{ t^2}2 \|\dot x(t)\|^2 + t^2b(t) \pa{ f (x(t))-\min_{\cH} f} -  \frac{s^2}{2} \|\dot  x(s)\|^2 - s^2b(s) \pa{f (x(s))-\min_{\cH} f} \\
&&+ (\alpha -1) \int_s^t\tau\|\dot x(\tau)\|^2d\tau 
- \int_s^t\pa{f (x(\tau))-\min_{\cH} f} \frac{d}{dt}\pa{\tau^2b(\tau)}	d\tau.
\end{eqnarray*}
Consequently, by setting 
\begin{eqnarray*}
{B(t)} 
&\eqdef& \frac{t^2}{2}\|\dot  x(t)\|^2 
+t^2b(t) \pa {f (x(t))-\min_{\cH} f} + (\alpha -1) \int_{t_0}^t\tau\|\dot x(\tau)\|^2d\tau\\
&& - \int_{t_0}^t\pa{f (x(\tau))-\min_{\cH} f} \frac{d}{dt}\pa{\tau^2b(\tau)} d\tau,
\end{eqnarray*}
we deduce that the function ${B(\cdot)} $ is non-increasing on $[t_0,+\infty[$. To ensure its convergence, we need to justify that ${B(t)}$  is bounded from below. First, the first three terms entering ${B(t)}$ are non-negative. We now use the condition $\frac{d}{dt} \left(t^2b(t)\right)\leq C_1tb(t)$ for $t$ large enough to deduce the existence of $t_1\geq t_0$ such that for all $t\geq t_1$
 \begin{eqnarray*}
{B(t)} \geq - C_1\int_{t_0}^\infty \tau b(\tau) \pa{f (x(\tau))-\min_{\cH} f} d\tau >-\infty,
\end{eqnarray*}
where we used statement~\ref{ACFR,rescale:claim4}.
Therefore, we have that ${B(t)}$ converges as $t\rightarrow +\infty$. Using  {again assertions} \ref{ACFR,rescale:claim4} and \ref{ACFR,rescale:claim5} and the hypothesis on $b$, we deduce the existence of  
\[
\ell \eqdef \lim_{t\rightarrow +\infty}\brac{\frac{t^2}{2} \|\dot x(t)\|^2 +t^2b(t)	\pa{f (x(t))-\min_{\cH}} f}\geq 0.
\] 
Suppose $\ell >0$, then there exists $t_2\geq t_1$ such that for every $t\geq t_2$
\begin{equation}\label{contradict-conv1}
\frac{t}{2} \|\dot  x(t)\|^2 
+tb(t)	\pa{f (x(t))-\min_{\cH} f} \geq \frac{\ell}{2t}.
\end{equation}
By integrating \eqref{contradict-conv1}, this leads to a contradiction with \ref{ACFR,rescale:claim4} and \ref{ACFR,rescale:claim5}.
We conclude that 
\[
\lim_{t\rightarrow \infty}\brac{\frac{t^2}{2} \|\dot x(t)\|^2 +t^2b(t)	\pa{f (x(t))-\min_{\cH} f}} = 0,
\]
which gives, as $t \to +\infty$ 
\[
f (x(t))-\min_{\cH} f ={\rm o} \pa{\dfrac{1}{ t^2b(t)}} \qandq \norm{\dot{x}(t)} ={\rm o} \pa{\dfrac{1}{ t}}.
\]
This completes the proof.
\end{proof}

%%%%%%%%%%%%%%%%%%%%%%%%%%%%%%
\subsection{Convergence of the trajectories}

Based on the previous Lyapunov analysis, and using Opial's lemma \cite{Opial} (which is a continuous time version of Lemma~\ref{le.Opial}), we now prove the following convergence result. Recall $w(\cdot)$ defined in \eqref{def-w}.

\begin{theorem} \label{ACFR,rescale_convergence_continu}
Take $\alpha > 1$. Let $\beta (\cdot)$ be a non-decreasing function. Assume that \ref{cond:C1}-\ref{cond:C2}-\ref{cond:C3} in Theorem {\rm \ref{ACFR,rescale}} hold.
Suppose moreover that
\begin{enumerate}[label={\rm($\cC_{\arabic*}$)},itemindent=10ex,start=4]
\item $\lim_{t \to +\infty} \dfrac{\beta (t)}{tw(t)}=0$, and \label{cond:C4} 
\item $\lim_{t \to +\infty} \dfrac{1}{t^2 w(t)}=0$. \label{cond:C5} 
\end{enumerate}
Let  $x: [t_0, +\infty[ \rightarrow \cH$ be a solution trajectory of $ \DINAVD{\alpha, \beta, b}$.  
Then,
\begin{enumerate}[label={\rm(\roman*)}]
\item for all $x^\star \in S $, the limit of $\norm{x(t) -x^\star}$ exists,  as $t\to +\infty$.
\item $x(t)$ converges weakly to a point in $S$, as $t\to +\infty$.
\end{enumerate}
\end{theorem}
\begin{proof}
The proof extends the Lyapunov analysis of Theorem~\ref{ACFR,rescale} by studying the convergence of  the anchor functions $t \mapsto \norm{x(t) -x^\star}^{2}$ for any $x^\star\in S$. Recall the Lyapunov function $E_\lambda(\cdot)$ in \eqref{def:E2}. We have seen that,  by choosing $\delta (t)=t^2 w(t)- (\lambda +1 -\alpha)t \beta (t)$ and $c= \lambda(\alpha-1-\lambda)$, the limit of
$E_{\lambda}(t)$ exists for both $\lambda=\alpha-1$ and $\lambda=\alpha-1-\varepsilon$, as $t\to + \infty$. In turn, the limit of the difference $E_{\alpha -1}(t)-E_{\alpha -1- \varepsilon}(t)$ also exists. Let us compute it. 
%We have
%%
%\begin{eqnarray*}
%&& E_{\alpha -1}(t)=t^2 w(t)(f(x(t))-f(x^\star))\\
%&&\hspace{1.8cm} + \frac{1}{2}\Vert (\alpha -1)(x(t)-x^\star)+t\Big(\dot{x}(t)+\beta(t)\nabla f(x(t))\Big)\Vert^{2} \\
%&& E_{\alpha -1-\varepsilon}(t)= \left(t^2 w(t) +\varepsilon t \beta (t)\right) (f(x(t))-f(x^\star))+ \frac{\varepsilon (\alpha -1-\varepsilon)}{2}\Vert x(t) -x^\star\Vert^{2} \\
%&&\hspace{1.8cm} + \frac{1}{2}\Vert  (\alpha -1)(x(t)-x^\star)+t\Big(\dot{x}(t)+\beta(t)\nabla f(x(t))\Big)-\varepsilon (x(t)-x^\star) \Vert^{2} .
%\end{eqnarray*}
%%
%
Straightforward calculation gives
\begin{multline*}
E_{\alpha-1-\varepsilon}(t) - E_{\alpha-1}(t) = \varepsilon t \beta (t)(f(x(t))-f(x^\star))+ \frac{\varepsilon (\alpha -1)}{2}\norm{x(t) -x^\star}^{2} \\
-\varepsilon \dotp{(\alpha -1)(x(t)-x^\star)+t\pa{\dot{x}(t)+\beta(t)\nabla f(x(t))}}{x(t) -x^\star} .
\end{multline*}
After simplification, we obtain that the limit of 
\begin{multline*}
E_{\alpha-1-\varepsilon}(t) - E_{\alpha-1}(t) = \varepsilon t \beta (t)(f(x(t))-f(x^\star)) -\frac{\varepsilon (\alpha -1)}{2}\norm{x(t) -x^\star}^{2} \\
-\varepsilon t \dotp{\pa{\dot{x}(t)+\beta(t)\nabla f(x(t))}}{x(t) -x^\star},
\end{multline*}
exists.
According to Theorem~\ref{ACFR,rescale}\ref{ACFR,rescale:claim1}, we have
\[
t \beta (t)(f(x(t))-f(x^\star)) \leq  t \beta (t)\dfrac{C_0}{t^2 w(t)}=C_0 \dfrac{\beta(t)}{t w(t)}.
\]
In view of hypothesis \ref{cond:C4}, the limit of the above expression is equal to zero.
Therefore, the limit as $t$ goes to infinity of 
\[
p(t)\eqdef \frac{\alpha -1}{2}\Vert x(t) -x^\star\Vert^{2} + t \left\dotp{\dot{x}(t)}{x(t) -x^\star \right} + t \beta(t) \left\dotp{\nabla f(x(t))}{x(t) -x^\star \right} 
\]
exists. From this, we want to show that the limit of $\norm{x(t) -x^\star}^{2}$ exists.
Set
\[
q(t)\eqdef \frac{\alpha -1}{2}\Vert x(t) -x^\star\Vert^{2} + (\alpha -1)\int_{t_0}^t \beta(s) \dotp{\nabla f(x(s))}{x(s) -x^\star} ds. 
\]
We have
\[
p(t)= q(t) + \frac{t}{\alpha -1}\dot{q}(t) - (\alpha -1)\int_{t_0}^t \beta(s) \left\dotp{\nabla f(x(s))}{x(s) -x^\star \right} ds.
\]
By Theorem~\ref{ACFR,rescale}\ref{ACFR,rescale:claim7}, we know that $\DS{\int_{t_0}^{+\infty} s w(s) \dotp{\nabla f(x(s))}{x(s) -x^\star} ds <+\infty.}$
As $\beta (s) =\mathcal O (sw(s))$ by \ref{cond:C4}, and since the integrand of this integral is non-negative by convexity, we deduce that the following limit exists:
\begin{equation}\label{basic100}
\lim_{t\to +\infty}  \int_{t_0}^t \beta(s) \left\dotp{\nabla f(x(s))}{x(s) -x^\star \right} ds.
\end{equation}
Therefore
\[
\lim_{t\to +\infty} \pa{q(t) + \dfrac{t}{\alpha - 1}\dot{q}(t)} 
\]
exists. Combining this with \cite[Lemma~7.2]{APR2}, since $\alpha > 1$, we deduce that the limit of $q(t)$ exists. Returning to the definition of $q(t)$ (which converges), and using again \eqref{basic100} and $\alpha > 1$, we finally obtain that, for any $x^\star \in S$  
\[
\lim_{t\to +\infty}  \norm{x(t) -x^\star} \; \mbox{ exists}.
\]
On the other hand, by Theorem~\ref{ACFR,rescale}\ref{ACFR,rescale:claim1} and assumption \ref{cond:C5}, we have
%, we have
%$
%0\leq f(x(t))-\min_{\cH} f \leq \dfrac{C_0}{t^2 w(t)}.
%$ 
%According to hypothesis $(\mathcal{C}_{5})$, we have 
%$\lim_{t \to +\infty} \dfrac{1}{t^2 w(t)}=0$. Therefore
\[
\lim_{t\to +\infty} f(x(t))=   \min_{\cH} f.
\]
Since $f$ is convex continuous, it is sequentially weakly lower semicontinuous. This implies that for any sequence $x(t_n)$ which converges weakly to some $\bar{x}$ as $t_n \to +\infty$, we have 
\[
f(\bar{x}) \leq \liminf f(x(t_n)) =  \min_{\cH} f ,
\]
and hence $\bar{x} \in S$. So all conditions of Opial's lemma \cite{Opial} are satisfied, which gives the weak convergence of the trajectories.
\end{proof}

\begin{remark}
Convergence of the trajectories has been proved for the weak topology of $\cH$. It is a natural question to ask whether one can obtain strong convergence. A counterexample due to Baillon \cite{Baillon} shows that the trajectories of the continuous steepest descent may converge weakly but not strongly. This example has been adapted by Attouch and Baillon in \cite{AB} to show that similar phenomenon occurs for the regularized Newton method. This suggests that convexity alone is not sufficient for the trajectories of $\DINAVD{\alpha, \beta, b}$ to strongly converge. However, adapting the arguments of \cite{ACPR} to the system $\DINAVD{\alpha, \beta, b}$,  we can reasonably expect this to be the case under certain geometrical or topological conditions on $f$. We do not elaborate more on this for the sake of brevity.
\end{remark}
 
%%%%%%%%%%%%%%%%%%%%%%%%%%%%%%
\subsection{Particular cases} 
Let us revisit the different cases considered in \cite{ACFR} in view of the convergence results obtained in Theorems~\ref{ACFR,rescale} and \ref{ACFR,rescale_convergence_continu}. They correspond to different choices for the parameters $\beta(\cdot)$ and $b(\cdot)$.

%%%
\paragraph*{Case~1} The system $\DINAVD{\alpha, \beta}$ corresponds to the values of the parameters $\beta(t) \equiv \beta$ and $b(t) \equiv 1$. In this case, $w(t)= 1- \frac{\beta}{t}$. Conditions \ref{cond:C1}, \ref{cond:C2} and \ref{cond:C3} are satisfied by taking $\alpha > 3$ and $ t> \frac{\alpha-2}{\alpha-3} \beta$. Conditions \ref{cond:C4} and \ref{cond:C5} are also obviously satisfied too. Therefore, as a corollary of Theorems \ref{ACFR,rescale} and 
\ref{ACFR,rescale_convergence_continu}, we obtain the following result appeared in \cite{APR2}.
\begin{corollary}[\cite{APR2}] \label{APR,2016}
Let  $x: [t_0, +\infty[ \rightarrow \cH$ be a trajectory of the dynamical system $\DINAVD{\alpha, \beta}$. Suppose $\alpha > 3$.  Then
\begin{enumerate}[label={\rm(\roman*)}]
\item $f(x(t))-  \min_{\cH}f ={\rm o} \pa{\dfrac{1}{t^2}}$;
\item  $\DS{\int_{t_0}^{\infty} t^2 \|\nabla f (x(t)) \|^2  dt} < + \infty$, and $\DS{\int_{t_0}^{+\infty} t \|\dot x(t)\|^2dt <+\infty}$;
\item $x(t)$ converges weakly to a point in $S$, as $t\to +\infty$.
\end{enumerate}
\end{corollary}
%
%%%
\paragraph*{Case~2} The system $\DINAVD{\alpha, \beta, 1+ \frac{\beta}{t}}$ corresponds to $\beta(t) \equiv \beta$ and $b(t) = 1+ \frac{\beta}{t}$. It was considered in \cite{SDJS}.
%, and is written as follows
%\begin{center}
%$ {\DINAVD{\alpha, \beta, 1+ \frac{\beta}{t}b}} \quad \ddot{x}(t) + \displaystyle{\frac{\alpha}{t} }\dot{x}(t) +  \beta  \nabla^2  f (x(t)) \dot{x} (t) + \left(1+ \frac{\beta}{t}\right) \nabla  f (x(t)) = 0.
%$
%\end{center}
Compared to $\DINAVD{\alpha, \beta}$, it has the additional coefficient $\frac{\beta}{t}$ in front of the gradient term. This vanishing coefficient will facilitate the computational aspects while keeping the structure of the dynamic. Observe that in this case, $w(t)\equiv 1$. Conditions \ref{cond:C1}, \ref{cond:C2} and \ref{cond:C3} boil down to $\alpha > 3$, while \ref{cond:C4} and \ref{cond:C5} are clearly satisfied. In this setting, Theorems~\ref{ACFR,rescale} and 
\ref{ACFR,rescale_convergence_continu} specialize to
\begin{corollary}  \label{APR,variant}
Let  $x: [t_0, +\infty[ \rightarrow \cH$ be a solution trajectory of the dynamical system $\DINAVD{\alpha, \beta, 1+ \frac{\beta}{t}}$. Suppose $\alpha > 3$. Then 
\begin{enumerate}[label={\rm(\roman*)}]
\item $f(x(t))-  \min_{\cH}f ={\rm o} \pa{\dfrac{1}{t^2}}$;
\item $\DS{\int_{t_0}^{\infty} t^2 \|\nabla f (x(t)) \|^2  dt}    < + \infty$, and $\DS{\int_{t_0}^{+\infty} t \|\dot x(t)\|^2dt <+\infty}$;
\item $x(t)$ converges weakly to a point in $S$, as $t\to +\infty$.
\end{enumerate}
\end{corollary}
%
%%%
\paragraph*{Case~3} The dynamical system \DINAVD{\alpha, 0, b}, which corresponds to $\beta(t) \equiv 0$, \ie no Hessian driven damping, was considered in \cite{ACR-rescale}. It comes  naturally from the temporal scaling of \AVD{\alpha}. 
%It is written as follows
%$$ {\DINAVD{\alpha, 0, b}} \quad \ddot{x}(t) + \displaystyle{\frac{\alpha}{t} }\dot{x}(t)  + b(t) \nabla  f (x(t)) = 0.
%$$
In this case, we have $w (t) = b(t)$. \ref{cond:C1} is equivalent to $b(t)> 0$ while \ref{cond:C2} becomes  
\begin{equation}\label{acb_C}
t\dot{b}(t)\leq (\alpha-3) b(t),
\end{equation}
which is precisely the condition introduced in \cite[Theorem~8.1]{ACR-rescale}. Under this condition, we have the convergence rate 
\[
f(x(t))-\min_{\cH}f =\cO\left(\dfrac{1}{t^{2}b(t)}\right) \, \mbox{ as } \, t \to + \infty.
\]
Note that the condition \eqref{acb_C} can be written as  $\frac{d}{dt} \left(t^2b(t)\right)\leq (\alpha-1)tb(t)$, which is exactly the condition ensuring (see Theorems~\ref{ACFR,rescale}\ref{ACFR,rescale:claim8})
\[
f(x(t))-\min_{\cH}f ={\rm o}\left(\dfrac{1}{t^{2}b(t)}\right) \, \mbox{ as } \, t \to + \infty.
\]
Of course, the interesting case corresponds to $\lim \dfrac{1}{t^{2}b(t)}=0$, which is nothing but condition \ref{cond:C5}. This makes clear the acceleration effect due to the time scaling. For  $b(t)= t^r$, we have $ f(x(t))-\min_{\cH}f ={\rm o}\pa{\dfrac{1}{t^{2+r}}}$, under the assumption $\alpha \geq 3+r$, \ie \eqref{acb_C} holds.
According to Theorem \ref{ACFR,rescale_convergence_continu}, under the stronger assumption
\[
 (\alpha-3) b(t) -t\dot{b}(t) \geq \varepsilon b(t),
\]
we obtain the weak convergence of the trajectories to optimal solutions.
%
%%%
\paragraph*{Case~4} Take $b(t) = ct^b$, $ \beta (t) = t^\beta$. We have $w (t) = ct^b-(\beta +1)t^{\beta -1}$ and  $\dot{w}(t) =cbt^{b-1}-(\beta ^2-1)t^{\beta -2}$. Conditions \ref{cond:C1}-\ref{cond:C2} amount respectively to assuming:
\begin{equation}\label{condgi}
ct^b>(\beta +1)t^{\beta -1}\text{ and } c(b-\alpha+3)t^b\leq (\beta +1)(\beta -\alpha +2)t^{\beta -1}.
\end{equation}
When $b=\beta -1$, the conditions \eqref{condgi} are equivalent to
$
\beta < c-1 \qandq \beta\leq \alpha -2,
$
which gives the convergence rate $f(x(t))-\min_{\cH} f = {\rm o} \left(\dfrac{1}{t^{\beta +1}}\right)$.\\
Let us now examine  the case $-1\leq \beta<\alpha-2$ and $b\in ]\beta-1,\alpha-3[$. Then  the conditions \eqref{condgi} are satisfied, since 
$$+\infty=\lim_{t\rightarrow+\infty}t^{b-\beta+1}>\beta+1\text{ and }-\infty=c(b-\alpha+3)\lim_{t\rightarrow+\infty}t^{b-\beta+1}<(\beta +1)(\beta -\alpha +2).
$$ 
Therefore,  $f(x(t))-\min_{\cH} f = {\rm o} \left(\dfrac{1}{t^{\beta +1}}\right)$. \\
The weak convergence of $x(t)$ to a minimizer follows from Theorem~\ref{ACFR,rescale_convergence_continu}  and the fact that with the above choice
\begin{eqnarray*}
\lim_{t \to +\infty} \frac{\beta (t)}{tw(t)}&=&\lim_{t \to +\infty} \frac{1}{ct^{b-\beta+1}-(\beta+1)} =0,\\
 \lim_{t \to +\infty} \dfrac{1}{t^2 w(t)}&=& \lim_{t \to +\infty} \frac{1}{t^{\beta+1}\left[ct^{b-\beta+1}-(\beta+1)\right]} =      0.
\end{eqnarray*}

%%%%%%%%%%%%%%%%%%%%%%%%%%%%%%%%%%%%%%%%%%%%%%%%%%%%%%%%%%%%%%%%%%%%%%%%%%%%%%%%%%%
\section{Convergence of proximal algorithms}\label{sec:prox}
%%%%%%%%%%%%%%%%%%%%%%%%%%%%%%%%%%%%%%%%%%%%%%%%%%%%%%%%%%%%%%%%%%%%%%%%%%%%%%%%%%%
%
Let us analyze the convergence properties of the proximal algorithms obtained by implicit temporal discretization of the continuous dynamic ${\DINAVD{\alpha, \beta, b}}$.
%that we recall below 
%$$ 
%\quad \ddot{x}(t) + \displaystyle{\frac{\alpha}{t} }\dot{x}(t) +  \beta (t) \nabla^2  f (x(t)) \dot{x} (t) + b(t)\nabla  f (x(t)) = 0.
%$$
We will show the convergence of the iterates generated by these algorithms, which complements the convergence rates of the values obtained in \cite{ACFR}. We take a fixed step size $h > 0,$  and denote by $x_k$ an approximation of $x(kh)$. To keep close to the continuous dynamic, we consider the following implicit scheme: $ k \geq 1,$  
\begin{multline}\label{dis1}
\frac{1}{h^{2}}\pa{x_{k+1}-2x_{k}+x_{k-1}}	+\frac{\alpha}{kh}\pa{\frac{1}{h}(x_{k+1}-x_{k})}\\
+\frac{\beta_{k}}{h}\bpa{\nabla f(x_{k+1})-\nabla f(x_{k})}+b_{k}\nabla f(x_{k+1})= 0.
\end{multline}
Indeed,  there are several other possibilities for the temporal discretization of $\dot{x}(t)$,  $\ddot{x}(t)$, and  $\nabla^2 f (x(t)) \dot{x} (t) = {\frac{d}{dt} \pa{\nabla  f (x(t))}}$. The study of these other discretization schemes are beyond the scope of this paper and deserves further work.

By rearranging the terms of \eqref{dis1}, the latter equivalently reads
\[
x_{k+1}+\frac{hk(\beta_{k}+hb_{k})}{k+\alpha} \nabla f(x_{k+1}) =x_{k}+ \frac{k}{k+\alpha}(x_{k}-x_{k-1})+ \frac{hk\beta_{k}}{k+\alpha} \nabla f(x_{k}).
\]
By setting $\lambda_{k}=\dfrac{hk(\beta_{k}+hb_{k})}{k+\alpha} $, $\alpha_k= \dfrac{k}{k + \alpha } $, we obtain the following iterative scheme:\\
%%%%%%%%%%%
\renewcommand{\algorithmcfname}{$\IPAHD$}
\begin{algorithm}[H]
\caption{Inertial Proximal Algorithm with Hessian Damping.}
Initialization: $x_0, x_1 \in \cH$ given\;
\For{$k=1,\ldots $}{
$y_{k} = x_{k}+ \alpha_k (x_{k}-x_{k-1})+ h \alpha_k \beta_{k} \nabla f(x_{k})$ \;  
$x_{k+1} = \prox_{\lambda_{k}f} (y_{k} )$.
}
\end{algorithm}
Note that a step of $\IPAHD$ involves both the calculation of a proximal term and of a gradient term relative to $f$. Thus, $\IPAHD$ could be  considered as a proximal-gradient algorithm as well. Since the proximal-gradient terminology is mainly used for composite problems involving smooth and non-smooth data, and here there is only one function $f$, $\IPAHD$ is called proximal. In addition, we will see that we can extend the algorithm to the case of a non-smooth  function $f$, in which case there are only proximal steps in the algorithm. 
For mathematical developments, it is convenient to use the following equivalent form of $\IPAHD$ (a direct consequence of \eqref{dis1}) 
\begin{equation}\label{basic-eq}
x_{k+1}-2x_{k}+x_{k-1}= -\pa{\frac{\alpha}{k}(x_{k+1}-x_{k})+h(\beta_{k}+hb_{k})\nabla f(x_{k+1})-h\beta_{k}\nabla f(x_{k})} .
\end{equation}
%
%%%%%%%%%%%%%%%%%%%%%%%%%%%%%
\subsection{Convergence rates}	
Let us first study the convergence of values for iterates generated by the algorithm $\IPAHD$. As in the continuous case, the Lyapunov analysis developed in the following theorem involves a positive parameter $\lambda$, which was taken equal to $\alpha - 1$ in \cite[Theorem~4]{ACFR}. The role of $\lambda$ will be central when it comes to the convergence proof of the iterates as we will see in the next section.
\begin{theorem} \label{ACFR,rescale-algo}
Suppose that $\alpha > 1$. Take $\lambda \in ]0,\alpha-1]$, set  $\gamma\eqdef\alpha -\lambda-1\geq 0$.
Define 
\begin{equation}\label{def:delta_b}
B_k \eqdef  k(hb_k+ \beta_{k} -  \beta_{k+1}) - \beta_{k+1}  \qandq \delta_k \eqdef  h\bpa{(k+1+\gamma)B_k + \gamma(k+1)\beta_{k+1}} ,
\end{equation}
and suppose that the following growth conditions are satisfied: 
\begin{enumerate}[label={\rm($\cG_{\arabic*}$)},itemindent=10ex]
\item $B_k > 0$; \label{cond:G1}
\item $\delta_{k+1} - \delta_k - h\lambda B_k\leq0$. \label{cond:G2}
\end{enumerate}
Then, $\delta_k$ is positive and, for any sequence $\seq{x_k}$ generated by $\IPAHD$, the following properties hold:
\begin{enumerate}[label={\rm(\roman*)}]
\item $0 \leq f(x_k)-\min_{\cH} f = \cO \pa{\dfrac{1}{\delta_k}}$ as  $k\to +\infty$; \label{ACFR,rescale-algo:claim1}
\item $\sum_{k \in \N} \bpa{\delta_k  - \delta_{k+1} + h\lambda B_k}\pa{f(x_{k+1})-\min_{\cH} f}<+\infty$; \label{ACFR,rescale-algo:claim2}
\item $\sum_{k \in \N} h^2\bpa{\frac{1}{2} B_k+ (k+1)\beta_{k+1}}B_k\norm{\nabla f(x_{k+1})}^2 <+\infty$; \label{ACFR,rescale-algo:claim3}
\item $\sum_{k \in \N}  k \| x_{k+1} - x_{k}\|^2 <+\infty$. \label{ACFR,rescale-algo:claim4}
\end{enumerate}
\end{theorem}

\begin{proof}
%{\bf Step 1:} 
%Take $\lambda \in ]0,\alpha-1]$. Set  $\gamma\eqdef\alpha -\lambda-1\geq 0$ and $c\eqdef \lambda\gamma$. 
Given $x^\star  \in S$, let us define 
\begin{eqnarray}
\cE_k(\lambda) &\eqdef&  \delta_k ( f (x_k)-  f(x^\star) ) +\frac{1}{2}\norm{v_k}^2 +\frac{c}2\|x_k-x^\star\|^2, \label{eq:lyapEk}\\
v_k &\eqdef& \lambda (x_k -x^\star) + k (x_{k} - x_{k-1} + \beta_k h \nabla f( x_{k})  ) ,
\end{eqnarray}
where $c$ is a non-negative parameter that will be adjusted in the course of the proof. For each $\lambda \in ]0,\alpha-1]$, $\cE_k(\lambda)$ is non-negative and we will show that $\cE_k(\lambda)$ can serve as a Lyapunov function, \ie we shall prove that $\seq{\cE_k(\lambda)}$ is a non-increasing sequence. 

\noindent We first have
\begin{multline}
 \cE_{k+1}(\lambda)  - \cE_k(\lambda)  = \pa{\delta_{k+1} - \delta_k}\pa{f (x_{k+1})- f(x^\star)}+ \delta_k  \pa{f (x_{k+1}) -f (x_k)} \\
+ \frac{1}{2}\pa{\|v_{k+1}\|^2 -\|v_k\|^2}+ \frac{1}{2}\pa{c\|x_{k+1}-x^\star\|^2 - c\|x_{k}-x^\star\|^2} .  \label{basic-formula-1-b}
\end{multline}
Let us first evaluate the third term in the right-hand side of \eqref{basic-formula-1-b}. Using successively the definition of $v_k$ and \eqref{basic-eq}, we first get
\begin{eqnarray*}
v_{k+1} - v_{k} 
&=& \lambda (x_{k+1} - x_{k}) +(k+1) \pa{x_{k+1} - x_{k} + \beta_{k+1} h \nabla f( x_{k+1})}\\
&&-k (x_{k} - x_{k-1} + \beta_k h \nabla f( x_{k}))\\
&=& (\lambda+1)  (x_{k+1} - x_{k})   +k (x_{k+1} - 2x_{k} +x_{k-1} )+ \beta_{k+1} h \nabla f( x_{k+1}) \\
&&+ h k\pa{\beta_{k+1}\nabla f( x_{k+1})-\beta_{k}\nabla f( x_{k})}    \\
&=& k \pa{x_{k+1} - 2x_{k} +x_{k-1}} + kh \beta_{k}\pa{\nabla f( x_{k+1})-\nabla f( x_{k})} \\
&&+ (\lambda+1)  (x_{k+1} - x_{k})   +  \beta_{k+1} h \nabla f( x_{k+1})+ kh(\beta_{k+1} - \beta_{k})\nabla f( x_{k+1})\\
&=& - b_k h^2 k \nabla f( x_{k+1}) - \alpha  (x_{k+1} - x_{k}) +(\lambda+1)  (x_{k+1} - x_{k})  \\
&& +\beta_{k+1} h \nabla f( x_{k+1})+ kh(\beta_{k+1} - \beta_{k})\nabla f( x_{k+1})\\
&=& (\lambda+1 - \alpha)  (x_{k+1} - x_{k})+ hk\pa{ \frac1k\beta_{k+1}  + \beta_{k+1} - \beta_{k} - hb_k }\nabla f( x_{k+1}) .
\end{eqnarray*}
Recalling the definition of $\gamma$ and $B_k$ in \eqref{def:delta_b}, we have
\begin{equation}\label{form_v_k}
v_{k+1} - v_{k} = -\gamma(x_{k+1} - x_{k}) - h B_k \nabla f( x_{k+1}) .
\end{equation}
Set $\Delta_k\eqdef\dfrac{1}{2}\pa{\|v_{k+1}\|^2 -\|v_k\|^2}$. Then with the simple identity
\[
\frac{1}{2}\|v_{k+1}\|^2 -\frac{1}{2}\|v_k\|^2  = \left\dotp{ v_{k+1} -v_{k}}{v_{k+1} \right} - \frac{1}{2}\|v_{k+1} - v_{k}\|^2  ,
\]
we obtain,
\begin{eqnarray*}
\Delta_k 
&=&  -\lambda \gamma\dotp{x_{k+1} -x^\star}{x_{k+1} - x_{k}}  - \gamma\pa{k+1+\frac{\gamma}2}
\norm{x_{k+1} - x_{k}}^2\\
&& - h\bpa{\gamma(k+1)\beta_{k+1} + (k+1+\gamma)B_k}\dotp{\nabla f( x_{k+1})}{x_{k+1} - x_{k}}\\
&&- h\lambda B_k\dotp{\nabla f( x_{k+1})}{x_{k+1} - x^\star}
- h^2\pa{\frac{B_k}{2}  + (k+1)\beta_{k+1}}B_k\norm{\nabla f( x_{k+1})}^2 .
\end{eqnarray*} 
Using the three-point identity
\[
\dotp{x_{k+1} -x^\star}{x_{k+1} - x_{k}}=\frac12\bpa{ \| x_{k+1} -x^\star \|^2-\| x_{k} -x^\star \|^2+\| x_{k+1} - x_{k}\|^2} ,
\]
we infer
\begin{align*}
&\cE_{k+1}(\lambda)  - \cE_k(\lambda) 
= \pa{\delta_{k+1} - \delta_k}\pa{f (x_{k+1})- f(x^\star)} + \delta_k  \pa{f (x_{k+1}) -f (x_k)} \\
&-\frac{\lambda \gamma}2\pa{\| x_{k+1} - x^\star \|^2 - \| x_{k} - x^\star \|^2 + \| x_{k+1} - x_{k}\|^2} - \gamma\pa{k + 1 + \frac{\gamma}2}\|x_{k+1}  -  x_{k}\|^2\\
& - h\bpa{\gamma(k+1)\beta_{k+1} + (k+1+\gamma)B_k}\dotp{\nabla f( x_{k+1})}{x_{k+1}  -  x_{k}} \\
& - h\lambda B_k\dotp{\nabla f( x_{k+1})}{x_{k+1}  -  x^\star} - h^2\pa{\frac{1}{2} B_k+ (k+1)\beta_{k+1}}B_k\norm{\nabla f( x_{k+1})}^2\\
& + \frac{1}{2}\bpa{c\|x_{k+1} - x^\star\|^2 - c\|x_{k} - x^\star\|^2}. 
\end{align*}
Taking $c=\lambda \gamma \geq 0$ and using the (convex) subdifferential inequality 
$\dotp{\nabla f( x_{k+1})}{x_{k+1} - x_{k}} \geq f( x_{k+1})-f( x_{k})$, we arrive at
\begin{align}
&\cE_{k+1}(\lambda)  - \cE_k(\lambda) 
\leq \pa{\delta_{k+1} - \delta_k}\pa{f (x_{k+1})- f(x^\star)} - h\lambda B_k\dotp{\nabla f( x_{k+1})}{x_{k+1}  -  x^\star}\nonumber\\ 
&+\pa{\delta_k -   h\bpa{\gamma(k+1)\beta_{k+1} + (k+1+\gamma)B_k}}\dotp{\nabla f( x_{k+1})}{x_{k+1}  -  x_{k}}  \nonumber\\ 
&- \frac{\gamma}2\pa{2k+\alpha+1} \| x_{k+1} - x_{k}\|^2- h^2\pa{\frac{1}{2} B_k +  (k+1)\beta_{k+1}}B_k\norm{\nabla f( x_{k+1})}^2.
\label{basic-formula-1-c}
 \end{align}
Using the definition of $\delta_k$ in \eqref{def:delta_b}, we have $\delta_k > 0$ under \ref{cond:G1}, and the second scalar product term in \eqref{basic-formula-1-c} is canceled. In view of \ref{cond:G1}, we use once again the convex subdifferential inequality
$\dotp{\nabla f( x_{k+1})}{x_{k+1} - x^\star} \geq f( x_{k+1})-f( x^\star)$, to obtain
\begin{multline}
\cE_{k+1}(\lambda)  - \cE_k(\lambda)  
\leq  \pa{\delta_{k+1} - \delta_k - h\lambda B_k}\pa{ f (x_{k+1})- f(x^\star)} \\
- h^2\pa{\frac{1}{2} B_k + (k+1) \beta_{k+1}}B_k\norm{\nabla f( x_{k+1})}^2  
- \frac{ \gamma}2\pa{2k+\alpha+1} \| x_{k+1} - x_{k}\|^2. \label{basic-formula-1-d}
\end{multline}
Since $\gamma \geq 0$ and under \ref{cond:G1}-\ref{cond:G2}, all terms in the bound \eqref{basic-formula-1-d} are non-positive, whence we deduce that the sequence $\seq{\cE_k(\lambda)}$ is non-negative and non-increasing. This yields that for all  $k \geq 0$,
\[
f(x_k)-\min_{\cH} f \leq \dfrac{{\cal E}_0(\lambda)}{\delta_k} ,
\]
hence proving claim \ref{ACFR,rescale-algo:claim1}. Assertions~\ref{ACFR,rescale-algo:claim2}-\ref{ACFR,rescale-algo:claim3}-\ref{ACFR,rescale-algo:claim4} follow by summing the inequality \eqref{basic-formula-1-d}.
\end{proof}

%%%%%%%%%%%%%%%%%%%%%%%%%%%%%
\subsection{Convergence of the iterates}
To prove convergence of the iterates, we will use the Lyapunov function $\cE_k(\lambda)$ in \eqref{eq:lyapEk} by appropriately choosing $\lambda$.
\begin{theorem}\label{thm:prox-conv-iter}
Let us make the same hypotheses as in Theorem~\ref{ACFR,rescale-algo}, and replace \ref{cond:G1}--\ref{cond:G2} respectively by: there exist $\varepsilon > 0$ and $\underline{B} > 0$ such that for all $k \geq 0$,
\begin{enumerate}[label={\rm($\cG_{\arabic*}^+$)},itemindent=10ex]
\item $B_k \geq \underline{B} > 0$; \label{cond:G1+}
\item $\delta_{k+1} - \delta_k - h\lambda B_k\leq -\varepsilon h B_k$. \label{cond:G2+}
\end{enumerate}
In addition, suppose that
\begin{enumerate}[label={\rm($\cG_{\arabic*}$)},itemindent=10ex,start=3]
\item $\DS \lim_{k \to +\infty} \frac{\beta_{k+1}}{B_k}=0$. \label{cond:G3}
\end{enumerate}
Then, the sequence $\seq{x_k}$ generated by the algorithm $\IPAHD$ converges weakly in $\cH$, and its limit belongs to $S=\argmin_{\cH} f$.
\end{theorem}
%%%%%%%%%%%%%%%%%%%
\begin{proof}
The proof relies on Opial's Lemma~\ref{le.Opial}. First, by Theorem~\ref{ACFR,rescale-algo}\ref{ACFR,rescale-algo:claim1}, for every $\lambda \in ]0,\alpha-1]$, we have  
\[
f(x_k)-\min_{\cH} f = \cO\pa{\dfrac{1}{\delta_k}}.
\]
Take $\lambda = \alpha -1$. Then $\gamma =0$ and $\delta_k= h(k+1)B_k$.
By assumption \ref{cond:G1+}, we deduce that $\delta_k \geq  h(k+1)\underline{B}$. Therefore, as $k\rightarrow +\infty$, we have $\delta_k \to +\infty$, and hence $f(x_k)\rightarrow\min_{\cH} f $. 
Since  $f$ is convex and continuous, it  is sequentially weakly lower semicontinuous.
Therefore, for every sequential weak cluster point $\bar x$ of $\seq{x_k}$, say $x_{k_j}\rightharpoonup \bar x$, we have
\[
f(\bar x) \leq \liminf_{j\rightarrow +\infty}f(x_{k_j}) =\lim_{k \rightarrow +\infty}f(x_k)=\min_{\cH} f .
\]
So, every sequential weak cluster point of $(x_k)_k$ belongs to $S$.

Now fix $x^\star\in S$, and we show that the sequence of the anchor functions $\seq{\norm{x_k -x^{\star}}}$ converges. According to the proof of Theorem~\ref{ACFR,rescale-algo}, we have that
for every $\varepsilon\in [0,\alpha -1[$,  the sequence $\seq{\cE_k(\alpha -1-\varepsilon)}$ converges. In turn, so does the sequence $\seq{\cE_k(\alpha -1-\varepsilon)-\cE_k(\alpha -1)}$.
Recalling that in the proof of Theorem~\ref{ACFR,rescale-algo}, $\gamma= \alpha -\lambda-1 $ and $c= \lambda \gamma$ are the values of the parameters which give the decreasing property of the sequence $\cE_k(\lambda)$, we have
\begin{eqnarray*}
&& \cE_k(\alpha -1-\varepsilon) =  h\pa{(k+1+ \varepsilon)B_k +  \varepsilon(k+1)\beta_{k+1}} \pa{f (x_k)-  f(x^\star)} \\
&& +\frac{1}{2}\norm{(\alpha -1-\varepsilon) (x_k -x^\star) + k (x_{k} - x_{k-1} + \beta_k h \nabla f( x_{k})  )}^2 +\frac{ \varepsilon (\alpha -1-\varepsilon)}{2} \|x_k-x^\star\|^2 \\
&&\cE_k(\alpha -1) =   h(k+1)B_k  ( f (x_k)-  f(x^\star) )\\
&&\hspace{1cm} +\frac{1}{2}\norm{(\alpha -1) (x_k -x^\star) + k (x_{k} - x_{k-1} + \beta_k h \nabla f( x_{k})  )}^2.
\end{eqnarray*}
Taking the difference, we obtain 
\begin{multline*}
\cE_k(\alpha -1-\varepsilon)- \cE_k(\alpha -1) =   h \varepsilon  \pa{B_k + (k+1) \beta_{k+1}}\pa{f (x_k)-  f(x^\star)}\\
 + \pa{\frac{ \varepsilon^2}{2}+\frac{ \varepsilon (\alpha -1-\varepsilon)}{2}} \|x_k-x^\star\|^2 \\
 - \varepsilon \dotp{(\alpha -1) (x_k -x^\star) + k (x_{k} - x_{k-1} + \beta_k h \nabla f( x_{k}))}{x_k -x^\star} .
\end{multline*}
Equivalently 
\begin{multline*}
\cE_k(\alpha -1-\varepsilon)- \cE_k(\alpha -1) =   h \varepsilon  \pa{B_k + (k+1) \beta_{k+1}} \pa{f (x_k)-  f(x^\star)}\\
+ \pa{\frac{ \varepsilon^2}{2}+\frac{ \varepsilon (\alpha -1-\varepsilon)}{2}- \varepsilon (\alpha -1)} \|x_k-x^\star\|^2 \\
- \varepsilon k\dotp{x_{k} - x_{k-1} + \beta_k h \nabla f( x_{k})}{x_k -x^\star} .
\end{multline*}
After reduction 
\begin{multline*}
\cE_k(\alpha -1-\varepsilon)- \cE_k(\alpha -1) =   h \varepsilon \pa{ B_k + (k+1) \beta_{k+1}}\pa{f (x_k)-  f(x^\star)}\\
+  \frac{ \varepsilon}{2}\pa{1-\alpha} \|x_k-x^\star\|^2 \\
- \varepsilon k\dotp{x_{k} - x_{k-1}}{x_k -x^\star} - \varepsilon k \beta_k h\dotp{\nabla f( x_{k})}{x_k -x^\star}.
\end{multline*}
Using the three-point identity
\[
2\dotp{ x_{k} -x^\star}{x_{k} -x_{k-1}} = \|x_{k} -x^\star\|^2 -\|x_{k-1} -x^\star\|^2 +\|  x_{k} -x_{k-1}\|^2,
\]
and after dividing by $ \varepsilon$, we deduce  the convergence of the sequence $\seq{\Delta_k}$ whose general term is given by 
\begin{multline}\label{eq:Dk}
 \Delta_k = h \pa{ B_k + (k+1) \beta_{k+1}} \pa{f (x_k)-  f(x^\star)} +   \frac{1}{2}\left(1-\alpha \right) \|x_k-x^\star\|^2 \\
-  k \beta_k h\dotp{ \nabla f( x_{k})}{x_k -x^\star}
-  \frac{k}{2}\bpa{\|x_{k} -x^\star\|^2 -\|x_{k-1} -x^\star\|^2 +\|  x_{k} -x_{k-1}\|^2}.
\end{multline}
To estimate the first term of $ \Delta_k $, we  use the upper-bound obtained in Theorem~\ref{ACFR,rescale-algo}\ref{ACFR,rescale-algo:claim1} by taking $\lambda =\alpha-1$, namely
\[
f(x_k)-\min_{\cH} f \leq \frac{1}{h(k+1)B_k}
\]
to obtain
\begin{eqnarray*}
 0\leq  h  \Big( B_k + (k+1) \beta_{k+1}\Big) ( f (x_k)-  f(x^\star) )  \leq \dfrac{B_k + (k+1) \beta_{k+1}}{(k+1)B_k} = \dfrac{1}{k+1}+ \frac{\beta_{k+1}}{B_k}.
\end{eqnarray*}
According to assumption \ref{cond:G3}, the right hand side of this inequality goes to zero as $k\to +\infty$, and thus so does the first term in \eqref{eq:Dk}.
In addition, we have $\DS\lim_{k} k \| x_{k} - x_{k-1}\|^2=0$, which follows from Theorem \ref{ACFR,rescale-algo}\ref{ACFR,rescale-algo:claim4}.
We have therefore obtained that the limit as $k\to +\infty$ of the sequence of real numbers $\seq{p_k}$ exists, where $p_k$ is defined by
\[
p_k \eqdef    \left(\alpha -1\right) \|x_k-x^\star\|^2 
+2k \beta_k h\left\dotp{ \nabla f( x_{k})}{x_k -x^\star   \right}
+ k\Big( \|x_{k} -x^\star\|^2 -\|x_{k-1} -x^\star\|^2 \Big).
\]
Let us set $u_k\eqdef\norm{x_k - x^{\star}}^2, w_k \eqdef h\beta_k \dotp{\nabla f( x_{k})}{x_{k} - x^\star}$ (the latter is non-negative by convexity of $f$). We have
\begin{eqnarray*}
p_k \eqdef  (\alpha -1)u_k + k(u_{k} - u_{k-1}) + 2kw_k .
\end{eqnarray*}
Set $q_k\eqdef (\alpha -1)\pa{u_k +2\sum_{i=0}^k w_i}$. We have 
\begin{equation}\label{basic-formula-2-a}
p_k= q_k+ \frac{k}{\alpha -1}\left(q_k-q_{k-1}\right) - 2(\alpha -1)\sum_{i=0}^kw_i .
\end{equation}
Let us now  prove the convergence of the series of non-negative real numbers
\[
\sum_k w_k = \sum_k h\beta_k \dotp{\nabla f( x_{k})}{x_{k} - x^\star} .
\]
Under conditions of Theorem~\ref{ACFR,rescale-algo}, we get from \eqref{basic-formula-1-c}
\begin{equation*}
\cE_{k+1}(\lambda)  - \cE_k(\lambda) \leq 
  \pa{\delta_{k + 1}  -  \delta_k}\pa{f (x_{k + 1})  -  f(x^\star)}  - h\lambda B_k\dotp{\nabla f( x_{k+1})}{x_{k+1}  -  x^\star} .
\end{equation*}
Using the (convex) differential inequality, we obtain, for each $1 \geq \rho>0$, 
\begin{multline*}
\bpa{(1-\rho)\lambda hB_k + \delta_k  - \delta_{k+1}}\pa{f (x_{k+1})- f(x^\star)} + \rho \lambda hB_k\dotp{\nabla f( x_{k+1})}{x_{k+1}  -  x^\star} \\
\leq
\cE_{k}(\lambda)  - \cE_{k+1}(\lambda) .
\end{multline*}
Recall the notations in the proof of Theorem~\ref{ACFR,rescale-algo}. Choosig $\lambda =\alpha-1>0$ and $\rho=\frac{\varepsilon}{\alpha-1} \in ]0,1[$, then $\gamma=0$ and $
\delta_k=(k+1)hB_k$. In view of condition \ref{cond:G2+}, we deduce that
\begin{eqnarray*}
(1-\rho)\lambda hB_k + \delta_k  - \delta_{k+1}&\geq& ( \alpha-1-\varepsilon)hB_k + \varepsilon hB_k-\lambda hB_k\\
&=&(\alpha-1 - \lambda)hB_k=0.
 \end{eqnarray*}
Therefore, 
\[
\rho \lambda hB_k\dotp{\nabla f( x_{k+1})}{x_{k+1} - x^\star} 
\leq  \cE_{k}(\alpha-1)  - \cE_{k+1}(\alpha-1) .
\]
By summing this inequality, we obtain
\[
\sum_khB_k\dotp{\nabla f( x_{k+1})}{x_{k+1} - x^\star} 
\leq
\frac{(\alpha-1){\cal E}_{0}(\alpha-1)}{\varepsilon} <+\infty .
\]
Returning to \eqref{basic-formula-2-a}, we have shown that the sequence  
\[
\seq{q_k+ \frac{k}{\alpha -1}\pa{q_k-q_{k-1}}}
\] 
converges to some limit. According to Lemma \ref{le.conv-seq} , we conclude that $\seq{q_k}$ converges to the same limit. By definition of $q_k$, and using that the series $\sum_k v_k$ converges, we conclude that the sequence $\seq{u_k}$ converges. All conditions of Opial's Lemma~\ref{le.Opial} are satisfied, which gives the weak convergence of the sequence $\seq{x_k}$ to some point in $S$.
\end{proof}

%%%%%%%%%%%%%%%%%%%%%%%%%%%%%%%%%%%%%%%%%%%%%%%%%%%%%%%%%%%%%%%%%%%%%%%%%%%%%%%%%%%
\subsection{Non-smooth case}\label{non-smooth-prox}

Now suppose that $f: \cH \to \Rb$ is a proper, lower semicontinuous and convex function.
To reduce to the smooth case, we follow the approach initiated in \cite{ACFR}. It  makes use of  the Moreau envelope of $f$, which is defined for any $ \theta> 0 $ by: 
\[
f_{\theta} (x) = \min_{z \in \cH} \left\lbrace f (z) + \frac{1}{2 \theta} \norm{ z - x}^2   \right\rbrace, \quad \text{for any $x\in \cH$.} 
\]
The interested reader may refer to \cite{BC,Bre1} for a comprehensive treatment of the Moreau envelope and its properties in a Hilbert setting. For instance, we recall that  $f_{\theta} $ is a continuously differentiable convex function, whose gradient is $\theta^{-1}$-Lipschitz continuous, and the set of minimizers is preserved by taking the Moreau envelope, that is $\argmin f_{\theta} = S = \argmin f$. {Owing to this property, the idea  is now to replace $f$ by $f_{\theta} $ in algorithm $\IPAHD$, and to take advantage of the continuous differentiability of $f_{\theta}$.}
The Hessian dynamic attached to $f_{\theta}$ would formally read
\[
\ddot{x}(t) + \frac{\alpha}{t} \dot{x}(t) 
+ \beta (t) \nabla^2 f_{\theta} (x(t))\dot{x} (t) + b(t)\nabla f_{\theta} (x(t)) = 0.
\]
However, we do not really need to work on this system (which requires $f_{\theta}$ to be $\cC^2$), but with the discretized form which only requires the function to be continuously differentiable, as is the case of $f_{\theta} $. Then, algorithm $\IPAHD$ now reads
\[
\begin{cases}
y_{k} &= x_{k}+ \dfrac{k}{k + \alpha }(x_{k}-x_{k-1})+{\dfrac{hk\beta_{k}}{k+\alpha }} \nabla f_{\theta}(x_{k})  \\
x _{k+1} &= \hbox{prox}_{ \lambda_{k}f_{\theta}} ( y_{k} ) ,
\end{cases}
\]
where we recall that $\lambda_{k}=\frac{hk(\beta_{k}+hb_{k})}{k+\alpha}$.
Thus, we just need to formulate these results in terms of $f$ and its proximal mapping. This is straightforward thanks to the following formulae from proximal calculus \cite{BC}:
\begin{enumerate}[label=$\bullet$]
\item $f_{\theta} (x)= f(\prox_{ \theta f}(x)) + \frac{1}{2\theta} \norm{x-\prox_{\theta f}(x)}^2$. 
\item $\nabla f_{\theta} (x)= \frac{1}{\theta} \left( x-\prox_{ \theta f}(x) \right)$,
\item $\prox_{ \lambda f_{\theta}}(x) = \frac{\theta}{\lambda +\theta}x + \frac{\lambda}{\lambda +\theta}\prox_{ (\lambda + \theta)f}(x).$
\end{enumerate}

\noindent We thus obtain the following  algorithm (NS stands for Non-Smooth):\\
\renewcommand{\algorithmcfname}{$\IPAHDNS$}
\begin{algorithm}[H]
\caption{Inertial Proximal Algorithm with Hessian Damping for Non-Smooth functions.}
Initialization: $x_0, x_1 \in \cH$ given\;
Set $\mu_k = \dfrac{\theta(k + \alpha)}{\theta(k + \alpha) + hk (\beta_k + hb_k )}$\;
\For{$k=1,\ldots $}{
$y_k=   x_{k} + \pa{1- \dfrac{\alpha}{k + \alpha}} (x_{k}  - x_{k-1})+
\dfrac{\beta h}{\theta}\pa{1- \dfrac{\alpha}{k + \alpha}}\pa{x_k-\prox_{ \theta f} (x_k)}$ \;  
$x_{k+1} =  \mu_k y_k + (1-\mu_k)\prox_{\frac{\theta}{\mu_k} f} (y_k)$.
}
\end{algorithm}

Capitalizing on the results of Theorem~\ref{ACFR,rescale-algo} and Theorem~\ref{thm:prox-conv-iter} and the properties of the Moreau envelope recapped above, we can now state the following convergence result:
\begin{theorem} 
Let $f: \cH \to \Rb$ be a proper, lower semicontinuous and convex function. Suppose that $\alpha > 1$, take $\theta > 0$ and $\lambda \in ]0,\alpha-1]$, set $\gamma\eqdef\alpha -\lambda-1\geq 0$, and $B_k$ and $\delta_k$ as in \eqref{def:delta_b}. Suppose that growth conditions \ref{cond:G1} and \ref{cond:G2} of Theorem~\ref{ACFR,rescale-algo} are satisfied.
Then, $\delta_k$ is positive and, for any sequence $\seq{x_k}$ generated by \IPAHDNS we have:
\begin{enumerate}[label={\rm (\roman*)}]
\item $f(\prox_{ \theta f}(x_k)) -\min_{\cH} f = \cO\pa{\dfrac{1}{\delta_k}}$ as $k\to +\infty$; \\
\item $\sum_{k \in \N}  \bpa{\delta_k  - \delta_{k+1} + h\lambda B_k}\pa{f(\prox_{ \theta f}(x_{k+1}))- f(x^\star)}<+\infty$; \\
\item $\sum_{k \in \N} h^2\pa{\dfrac{B_k}{2} + (k+1)\beta_{k+1}}B_k\norm{x_{k+1} - \prox_{\theta f}(x_{k+1})}^2 <+\infty$; \\
\item $\sum_{k \in \N} k \norm{x_{k+1} - x_{k}}^2 <+\infty$.
\end{enumerate}
Assume moreover that conditions \ref{cond:G1+}, \ref{cond:G2+} and \ref{cond:G3} of Theorem~\ref{thm:prox-conv-iter} are verified. Then any sequence $\seq{x_k}$ generated 
by algorithm \IPAHDNS converges weakly in $\cH$, and its limit belongs to $S$.
\end{theorem}

%%%%%%%%%%%%%%%%%%%%%%%%%%%%%%%%%%%%%%%%%%%%%%%%%%%%%%%%%%%%%%%%%%%%%%%%%%%%%%%%%%%
\section{Convergence of gradient algorithms}\label{sec:gradient}
%%%%%%%%%%%%%%%%%%%%%%%%%%%%%%%%%%%%%%%%%%%%%%%%%%%%%%%%%%%%%%%%%%%%%%%%%%%%%%%%%%%
%
Throughout this section, $ f : \cH \to \R$ is a convex differentiable function whose gradient is $L$-Lipschitz continuous. In line with \cite {ACPR,CD,ACFR}, our analysis is based on the inertial dynamic $\DINAVD{\alpha, \beta, 1+ \frac{\beta}{t}}$ with damping parameters $\alpha \geq 3$, $\beta \geq 0$, that we recall for convenience:
\[
\ddot{x}(t) + \DS{\frac{\alpha}{t} }\dot{x}(t) +  \beta  \dfrac{d}{dt} \nabla f (x(t)) + \left(1+ \frac{\beta}{t}\right) \nabla  f (x(t)) = 0.
\]
%Indeed, the presence of the coefficient $\left(1+ \beta/t\right)$ (which tends to $1$)  makes the calculation a little easier.
Consider the following time discretization of $\DINAVD{\alpha, \beta, 1+ \frac{\beta}{t}}$ with  fixed temporal step size $h>0$,  and where we set $s=h^2$:
\begin{eqnarray*}
&&\frac{1}{s}(x_{k+1} - 2x_{k}+ x_{k-1} ) +    \frac{\alpha}{ks}(x_{k} - x_{k-1}) +  \frac{\beta}{\sqrt{s}}(\nabla f( x_{k}) - \nabla f( x_{k-1})   ) \\
&&+ \frac{\beta}{k\sqrt{s}} \nabla f( x_{k-1})  +  
\nabla f( y_{k}) =0.
\end{eqnarray*}
Taking $y_k$ inspired by the Nesterov accelerated gradient method, we obtain the following algorithm:\\
\renewcommand{\algorithmcfname}{$\IGAHD$}
\begin{algorithm}[H]
\caption{Inertial Gradient Algorithm with Hessian Damping.}
Initialization: $x_0, x_1 \in \cH$ given\;
Set $\alpha_k \eqdef  1 - \frac{\alpha}{k}$\;
\For{$k=1,\ldots $}{
$y_k = x_{k} + \alpha_k ( x_{k}  - x_{k-1}) - \beta \sqrt{s} \pa{ \nabla f (x_k)  - \nabla f (x_{k-1})} -  \frac{\beta\sqrt{s}}{k}\nabla f(x_{k-1})$ \;  
$x_{k+1} =  y_k - s \nabla f (y_k)$.
}
\end{algorithm}
%
%%%%%%%%%%%%%%%%%%%%%%%%%%%%%
\subsection{Convergence rates}
Following \cite{AC2}, set $t_{k+1} =\frac{k}{\alpha-1}$, whence $t_k = 1 + t_{k+1} \alpha_k$. Given  $x^\star \in S = \argmin f$, our Lyapunov analysis is based on the sequence $\seq{E_k}$ 
\begin{align}
E_k &\eqdef  t_k^2( f (x_k)-  f( x^\star) ) +\frac{1}{2s}\norm{v_k}^2 \label{Lyap-function-grad1}\\
v_k &\eqdef  (x_{k-1}  - x^\star) + t_k\pa{x_{k} - x_{k-1} + \beta \sqrt{s} \nabla f( x_{k-1})}.\label{Lyap-function-grad2}
\end{align}

\begin{theorem}\label{pr.decay_E_k}
Let  $f: \cH \to \R$ be a convex function whose gradient is $L$-Lipschitz continuous. Let $\seq{x_k}$ be a sequence generated by algorithm {\IGAHD}, where $\alpha \geq 3$, $0 \leq \beta <  2 \sqrt{s}$ and $sL \leq 1$. Then the sequence  $\seq{E_k}$  defined by {\rm\eqref{Lyap-function-grad1}--\eqref{Lyap-function-grad2}} is non-increasing, and the following convergence rates are satisfied:
\begin{enumerate}[label={\rm (\roman*)}]
\item $f(x_k)-\min_{\cH} f = \cO\pa{\dfrac{1}{k^2}}$ as $k\to +\infty$; \label{pr.decay_E_k:claim1}
\item Suppose moreover that $\alpha >3$. Then $\sum_{k \in \N} k\pa{f(x_k)-\min_{\cH} f}<+\infty$; \label{pr.decay_E_k:claim2}
\item Suppose that $\beta >0$. Then
\[
\sum_{k \in \N}  k^2 \| \nabla f (y_k) \|^2 < +\infty \qandq \sum_{k \in \N}  k^2 \| \nabla f (x_k) \|^2 < +\infty. \label{pr.decay_E_k:claim3}
\]
\end{enumerate}
\end{theorem}
\begin{proof}
The proof is an adaptation of \cite{ACFR}. 
%It is preparatory to the convergence of the iterates, which is more involved and where we will used a parametrized family of Lyapunov functions  $\seq{E_k (\lambda)}$.
We rely on the following reinforced version of the descent lemma, see \cite[Lemma~1]{ACFR}, and which is specific to the convex case. Since $s \leq \frac{1}{L}$, and $\nabla f$ is convex and $L$-Lipschitz continuous,
\begin{equation}\label{eq:descentlemma}
f(y - s \nabla f (y)) \leq f (x) + \left\dotp{ \nabla f (y)}{y-x \right} -\frac{s}{2} \|  \nabla f (y) \|^2 -\frac{s}{2} \| \nabla f (x)- \nabla f (y) \|^2
\end{equation}
for all $x, y\in \cH$. Let us write it successively at $y=y_k$ and  $x= x_k$, then at $y=y_k$,  $x= x^\star$. Since $\xkp=y_k-s\nabla f (y_k)$ and $\nabla f (x^\star)=0$, we get
\begin{align}
f(x_{k+1}) &\leq f(x_k) + \left\dotp{  \nabla f(y_k)}{y_k-x_k \right} -\frac{s}{2} \|   \nabla f (y_k) \|^2 -\frac{s}{2} \| \nabla f (x_k)- \nabla f (y_k) \|^2 \label{eq.majo_Theta(x_k+1)-Theta(x_k)-a}\\
f(x_{k+1}) &\leq f(x^\star) + \left\dotp{  \nabla f (y_k)}{y_k-x^\star \right} -\frac{s}{2} \|  \nabla f (y_k) \|^2 -\frac{s}{2} \| \nabla f (y_k) \|^2. \label{eq.majo_Theta(x_k+1)-min_Theta-b}
\end{align}
Multiplying \eqref{eq.majo_Theta(x_k+1)-Theta(x_k)-a} by $t_{k+1}-1$ (which is non-negative for $k$ large enough), then adding \eqref{eq.majo_Theta(x_k+1)-min_Theta-b}, we derive that
\begin{multline}\label{eq.combinaison_eq_prec_2-a}
t_{k+1}(f(x_{k+1}) -f(x^\star))\leq (t_{k+1}-1)(f(x_k)-f(x^\star)) \\
+\dotp{ \nabla f (y_k)}{(t_{k+1}-1)(y_k-x_k)+y_k-x^\star} -\frac{s}{2}t_{k+1}\| \nabla f (y_k) \|^2 \\
- \frac{s}{2}(t_{k+1}-1) \| \nabla f (x_k)- \nabla f (y_k) \|^2 - \frac{s}{2} \| \nabla f (y_k) \|^2.
\end{multline}
Let us multiply \eqref{eq.combinaison_eq_prec_2-a} by $t_{k+1}$ to make appear $E_k$. We obtain 
\begin{multline}
t_{k+1}^2(f(x_{k+1}) -f(x^\star))\leq (t_{k+1}^2-t_{k+1}-t_k^2)(f(x_k)-f(x^\star)) + t_k^2(f(x_k)-f(x^\star)) \\
+t_{k+1} \dotp{ \nabla f (y_k)}{(t_{k+1}-1)(y_k-x_k)+y_k-x^\star} -\frac{s}{2}t_{k+1}^2\| \nabla f (y_k) \|^2\\
- \frac{s}{2}(t_{k+1}^2-t_{k+1}) \| \nabla f (x_k)- \nabla f (y_k) \|^2 -
\frac{s}{2}t_{k+1} \| \nabla f (y_k) \|^2. \label{basic-est-1}
\end{multline}
We first assume that $\alpha \geq 3$. So we have $t_{k+1}^2-t_{k+1}-t_k^2 \leq 0$, which gives 
\begin{multline*}
t_{k+1}^2(f(x_{k+1} -f(x^\star))\leq  t_k^2(f(x_k)-f(x^\star)) \\
+t_{k+1} \dotp{ \nabla f (y_k)}{(t_{k+1}-1)(y_k-x_k)+y_k-x^\star} -\frac{s}{2}t_{k+1}^2\| \nabla f (y_k) \|^2 \\
- \frac{s}{2}(t_{k+1}^2-t_{k+1}) \| \nabla f (x_k)- \nabla f (y_k) \|^2 -
\frac{s}{2}t_{k+1} \| \nabla f (y_k) \|^2.
\end{multline*}
%The central issue is to write the above  line in a recursive form.
According to the definition of $E_k$, we infer
\begin{eqnarray*}
E_{k+1}  - E_k &\leq& t_{k+1} \dotp{ \nabla f (y_k)}{(t_{k+1}-1)(y_k-x_k)+y_k-x^\star} -\frac{s}{2}t_{k+1}^2\| \nabla f (y_k) \|^2\\
&& - \frac{s}{2}(t_{k+1}^2-t_{k+1}) \| \nabla f (x_k)- \nabla f (y_k) \|^2 -
\frac{s}{2}t_{k+1} \| \nabla f (y_k) \|^2 \\
&& + \frac{1}{2s}\|v_{k+1}\|^2 -\frac{1}{2s}\|v_k\|^2 .
\end{eqnarray*}
Let us compute the last term in the last expression with the help of the elementary identity
\[
\frac{1}{2}\|v_{k+1}\|^2 -\frac{1}{2}\|v_k\|^2  = \left\dotp{ v_{k+1} -v_{k}}{v_{k+1} \right} - \frac{1}{2}\|v_{k+1} - v_{k}\|^2  .
\]
By definition of $v_k$ in \eqref{Lyap-function-grad2}, and according to \IGAHD and $t_k - 1= t_{k+1} \alpha_k$, we have
\begin{align*}
v_{k+1} - v_{k} 
&= x_{k} - x_{k-1} + t_{k+1}(x_{k+1} - x_{k} + \beta \sqrt{s} \nabla f( x_{k})) -t_k (x_{k} - x_{k-1} + \beta \sqrt{s}  \nabla f( x_{k-1}))\\
&= t_{k+1}(x_{k+1} - x_{k}) -(t_k -1)  (x_{k} - x_{k-1})  +\beta \sqrt{s} \bpa{ t_{k+1} \nabla f( x_{k})-  t_{k}\nabla f( x_{k-1})}  \\
&= t_{k+1}\Big(x_{k+1} - (x_{k} +\alpha_k  (x_{k} - x_{k-1})\Big)   +\beta \sqrt{s} \bpa{t_{k+1} \nabla f( x_{k})-  t_{k}\nabla f( x_{k-1})} \\
&= t_{k+1}\left(x_{k+1} - y_k\right) -  t_{k+1}\beta \sqrt{s}( \nabla f( x_{k})-  \nabla f( x_{k-1}) ) - t_{k+1} \frac{\beta \sqrt{s}}{k}\nabla f( x_{k-1})\\
&\quad+\beta \sqrt{s} \bpa{t_{k+1} \nabla f( x_{k})-  t_{k}\nabla f( x_{k-1})}\\
&= t_{k+1}(x_{k+1} - y_k) +  \beta \sqrt{s} \pa{t_{k+1} \pa{1- \frac{1}{k}} -   t_{k}} \nabla f( x_{k-1}) \\
&= t_{k+1}\left(x_{k+1} - y_k\right) = - st_{k+1} \nabla f( y_{k})  .
\end{align*}
Hence
\begin{multline*}
\frac{1}{2s}\|v_{k+1}\|^2 -\frac{1}{2s}\|v_k\|^2  =  - \frac{s}{2}t_{k+1}^2\| \nabla f( y_{k}) \|^2 \\
- t_{k+1} \dotp{\nabla f( y_{k})}{x_{k} - x^\star  +  t_{k+1}\pa{x_{k+1} - x_{k} + \beta \sqrt{s} \nabla f( x_{k})}} .
\end{multline*}
Collecting the above results, we obtain
\begin{eqnarray*}
 E_{k+1}  - E_k &\leq& t_{k+1} \dotp{\nabla f (y_k)}{(t_{k+1}-1)(y_k-x_k)+y_k-x^\star} -st_{k+1}^2\| \nabla f (y_k) \|^2\\
&&  - t_{k+1} \dotp{\nabla f( y_{k})}{x_{k} - x^\star  +  t_{k+1}\pa{x_{k+1} - x_{k} + \beta \sqrt{s} \nabla f( x_{k})}} \\
 && - \frac{s}{2}(t_{k+1}^2-t_{k+1}) \| \nabla f (x_k)- \nabla f (y_k) \|^2 -
\frac{s}{2}t_{k+1} \| \nabla f (y_k) \|^2 .
 \end{eqnarray*}
Equivalently
\begin{multline*}
E_{k+1}  - E_k \leq t_{k+1} \dotp{\nabla f (y_k)}{A_k } -st_{k+1}^2 \| \nabla f (y_k) \|^2 \\
 - \frac{s}{2}(t_{k+1}^2-t_{k+1}) \| \nabla f (x_k)- \nabla f (y_k) \|^2 -
\frac{s}{2}t_{k+1} \| \nabla f (y_k) \|^2 ,
\end{multline*}
with
\begin{align*}
A_k
&= (t_{k+1}-1)(y_k-x_k)+y_k  - x_{k}  -  t_{k+1}\pa{x_{k+1} - x_{k} + \beta \sqrt{s} \nabla f( x_{k})}\\
&= t_{k+1}(y_k - x_k) - t_{k+1}(x_{k+1} - x_{k})-  t_{k+1} \beta \sqrt{s} \nabla f( x_{k})\\
&= t_{k+1}(y_k - x_{k+1} )-  t_{k+1} \beta \sqrt{s} \nabla f( x_{k})\\
&= st_{k+1}\nabla f (y_k) -  t_{k+1} \beta \sqrt{s} \nabla f( x_{k}).
\end{align*}
Consequently
\begin{align*}
E_{k+1}  - E_k 
&\leq t_{k+1} \dotp{\nabla f (y_k)}{st_{k+1}\nabla f (y_k) -  t_{k+1} \beta \sqrt{s} \nabla f( x_{k})}\\
&\quad -st_{k+1}^2 \| \nabla f (y_k) \|^2   - \frac{s}{2}(t_{k+1}^2-t_{k+1}) \| \nabla f (x_k)- \nabla f (y_k) \|^2 - \frac{s}{2}t_{k+1} \| \nabla f (y_k) \|^2 \\
&= - t_{k+1}^2 \beta \sqrt{s}\dotp{\nabla f (y_k)}{\nabla f( x_{k})}  - \frac{s}{2}(t_{k+1}^2-t_{k+1}) \| \nabla f (x_k)- \nabla f (y_k) \|^2 \\
&\quad - \frac{s}{2}t_{k+1} \| \nabla f (y_k) \|^2 \\
&=   - t_{k+1}B_k ,
\end{align*}
where
\[
B_k \eqdef t_{k+1} \beta \sqrt{s}\dotp{\nabla f (y_k)}{\nabla f( x_{k})}  + \frac{s}{2}(t_{k+1}-1) \| \nabla f (x_k)- \nabla f (y_k) \|^2 + \frac{s}{2} \| \nabla f (y_k) \|^2 .
\]
When $\beta=0$ we have $B_k\geq 0$. Let us analyze the  sign of $B_k$ in the case $\beta>0$. Denote for short $X = \nabla f (x_k)$ and $Y = \nabla f (y_k)$. We have
\begin{align*}
B_k 
& = \frac{s}{2} \|Y\|^2 + \frac{s}{2}(t_{k+1}-1) \|Y-X\|^2 + t_{k+1} \beta \sqrt{s}\dotp{Y}{X}  \\
& = \frac{s}{2}t_{k+1} \| Y \|^2 + \pa{t_{k+1}(\beta \sqrt{s}-s)+s} \dotp{Y}{X} + \frac{s}{2}(t_{k+1}-1) \|X\|^2 \\
& \geq  \frac{s}{2}t_{k+1} \|Y\|^2
 - \pa{t_{k+1}(\beta \sqrt{s}-s)+s}\|Y\| \|X\| + \frac{s}{2}(t_{k+1}-1) \|X\|^2.
\end{align*}
Elementary algebra gives that the above quadratic form is non-negative when
\[
\left(t_{k+1}( \beta \sqrt{s} -s  ) +s  \right)^2 \leq s^2 t_{k+1} (t_{k+1}-1).
\]
Recall that $t_k$ is of order $k$. Hence, this inequality is satisfied for $k$ large enough if  $( \beta \sqrt{s} -s  )^2 < s^2$, which is equivalent to $\beta < 2 \sqrt{s}$.
Under this condition, we deduce that $E_k$ is non-increasing which entails claim~\ref{pr.decay_E_k:claim1}. Similar argument gives that for $0 <  \varepsilon <  2 \sqrt{s}\beta - \beta^2$ (such $ \varepsilon$ exists according to assumption $0 < \beta < 2 \sqrt{s}$) 
\[
E_{k+1} - E_k +  \demi  \varepsilon t_{k+1}^2 \| \nabla f (y_k) \|^2 \leq 0. 
\]
Summing these inequalities, we obtain assertion~\ref{pr.decay_E_k:claim3} on $y_k$. The summability claim on $x_k$ is consequence of that on $y_k$. Indeed, Lipschitz continuity of the gradient and \IGAHD yield
\[
\norm{\nabla f(x_{k+1})} \leq (1+Ls)\norm{\nabla f(y_k)} .
\]
Suppose now $\alpha >3$. Returning to \eqref{basic-est-1}, by a similar argument as above we obtain
\begin{eqnarray}
E_{k+1}  - E_k +   (t_k^2 -t_{k+1}^2+t_{k+1})\pa{f(x_k)-f(x^\star)} \leq 0. \label{basic-est-2}
\end{eqnarray}
From $t_{k+1} =\frac{k}{\alpha-1}$ we get $t_k^2 -t_{k+1}^2+t_{k+1}\geq \frac{\alpha-3}{(\alpha-1)^2} k$, and it follows from \eqref{basic-est-2} that
\begin{eqnarray}
E_{k+1}  - E_k +   \frac{\alpha-3}{(\alpha-1)^2} k\Big(f(x_k)-f(x^\star)\Big) \leq 0. \label{basic-est-3}
\end{eqnarray}
By summing the inequalities \eqref{basic-est-3} with respect to $k$, and since $\alpha > 3$, 
we obtain assertion~\ref{pr.decay_E_k:claim2}.
\end{proof}

%%%%%%%%%%%%%%%%%%%%%%%%%%%%%
\if
{
\subsection{Convergence of the iterates}
Let us write the algorithm \IGAHD equivalently as follows
\begin{eqnarray}\label{pret1}
\begin{cases}
y_k=   x_{k} + \alpha_k ( x_{k}  - x_{k-1}) -
g_k \vspace{2mm} 
\\
x_{k+1} = y_k - s \nabla f (y_k) .
\end{cases}
\vspace{1mm}
\end{eqnarray}
where
\begin{equation}\label{pret2}
g_k = \beta \sqrt{s}   \left( \nabla f (x_k)  - \nabla f (x_{k-1}) \right) +  \frac{\beta \sqrt{s}}{k}\nabla f( x_{k-1}).
\end{equation}
According to the estimates of the gradient terms obtained in Theorem \ref{pr.decay_E_k}, let us show that \IGAHD falls into the setting of the perturbed version of $\mbox{(AVD)}_{\alpha}$ studied in \cite{ACPR}.
Indeed, this is not immediate, and requires some transformation to exactly fit the setting of \cite{ACPR}.
First observe that, according to Theorem \ref{pr.decay_E_k}, we have

}
\fi

%%%%%%%%%%%%%%%%%%%%%%%%%%%%%
\subsection{Convergence of the velocities}
Here, we provide some key estimates on fast convergence of the velocities, which will also prove useful when it will come to guaranteeing convergence of the sequence of iterates.
\begin{theorem}\label{pr.decay_W_k}
Let  $f: \cH \to \R$ be a convex function whose gradient is $L$-Lipschitz continuous. Let $\seq{x_k}$ be a sequence generated by {\IGAHD}, where $\alpha > 3$, $0 < \beta <  2 \sqrt{s}$ and $sL \leq 1$. Then the following convergence rates on the velocities are satisfied:
\begin{enumerate}[label={\rm (\roman*)}]
\item $\sup_k  k\|  x_{k}  - x_{k-1} \|<+\infty$; \label{pr.decay_W_k:claim1}
\item $\sum_k  k\|  x_{k}  - x_{k-1} \|^2<+\infty$. \label{pr.decay_W_k:claim2}
\end{enumerate}
\end{theorem}
\begin{proof}
%Let us establish a descent property for the global energy
%$$
%W_k\eqdef f(x_k) + \frac{1}{2s} \| x_k-x_{k-1} \|^2 .
%$$
From the reinforced descent lemma, see \eqref{eq:descentlemma}, we have
\begin{equation}\label{energy02}
f(y - s \nabla f (y)) +  \frac{1}{2s} \| y-x-s \nabla f (y) \|^2  \leq f (x)+\frac{1}{2s} \| y-x \|^2 .
\end{equation}
Evaluating \eqref{energy02} at $y=y_k$ and $x=x_k$, we obtain
\begin{equation}\label{energy1}
f(x_{k+1})+ \frac{1}{2s} \| x_{k+1} -x_{k} \|^2\leq f (x_k)  +\frac{1}{2s} \| y_k-x_k \|^2.
\end{equation}
Let us estimate this last term. According to the definition of $y_k$, and using the monotonicity of $\nabla f$, and Cauchy-Schwarz inequality, we obtain
\begin{align*}
&\| y_k-x_k \|^2 =   \| \alpha_k ( x_{k}  - x_{k-1}) -
\beta \sqrt{s}   \left( \nabla f (x_k)  - \nabla f (x_{k-1}) \right) -  \frac{\beta \sqrt{s}}{k}\nabla f( x_{k-1}) \|^2\\
&\leq \alpha_k ^2\|  x_{k}  - x_{k-1}\|^2 +
 \beta^2 s \| \nabla f (x_k)  - \nabla f (x_{k-1})  \|^2 +  \frac{\beta^2 s}{k^2}\|\nabla f( x_{k-1}) \|^2\\
&\quad + \frac{2\beta \sqrt{s}}{k}\|  x_{k}  - x_{k-1}\|\|\nabla f( x_{k-1}) \| + \frac{2\beta^2 s}{k}\| \nabla f (x_k)  - \nabla f (x_{k-1})  \| \|\nabla f( x_{k-1}) \|.
\end{align*}
To lighten notation, set $g_k \eqdef \|\nabla f (x_k) \|+\| \nabla f (x_{k-1}) \|$. We get for $k \geq 1$
\begin{align*}
\|y_k-x_k \|^2 
&\leq \alpha_k ^2\|  x_{k}  - x_{k-1}\|^2 + \beta^2 s g_k^2 +  \dfrac{\beta^2 s}{k^2}g_{k}^2 + \dfrac{2\beta \sqrt{s}}{k}\|  x_{k}  - x_{k-1}\|g_{k} + \dfrac{2\beta^2 s}{k}g_{k}^2\\
&\leq \alpha_k ^2\|  x_{k}  - x_{k-1}\|^2 + \dfrac{2\beta \sqrt{s}}{k}\|  x_{k}  - x_{k-1}\|g_{k} + 4\beta^2 s g_{k}^2.
\end{align*}
Plugging this into \eqref{energy1}, we obtain
\begin{equation*}
f(x_{k+1})+ \frac{1}{2s} \| x_{k+1} -x_{k} \|^2 \leq f (x_k)  +\frac{\alpha_k ^2}{2s} \|  x_{k}  - x_{k-1}\|^2 
+ \frac{\beta}{k\sqrt{s}}\|  x_{k}  - x_{k-1}\|g_{k} + 2\beta^2  g_{k}^2. 
\end{equation*}
Write for short $\theta_k \eqdef f (x_k)  -f(x^\star)$ and $d_k\eqdef \frac{1}{2s} \| x_{k} -x_{k-1} \|^2 $. Recalling $\alpha_k = \frac{k-\alpha}{k}$, we get
\begin{equation*}
\theta_{k+1} +  d_{k+1} \leq \theta_k + \frac{(k-\alpha)^2}{k^2} d_k + \frac{\beta}{k\sqrt{s}}\|  x_{k}  - x_{k-1}\|g_{k} + 2\beta^2  g_{k}^2 .
\end{equation*}
After multiplication by $k^2$ we obtain
\begin{equation*}
 k^2 d_{k+1} -(k-\alpha)^2 d_k \leq k^2 \left(\theta_k  -\theta_{k+1} \right) 
+ \frac{\beta}{\sqrt{s}}k\|x_{k} - x_{k-1} \|g_{k} + 2\beta^2   k^2 g_{k}^2.
\end{equation*}
Let us write the above expression in a recursive form
\begin{multline}
k^2 d_{k+1} - (k-1)^2 d_k  + (\alpha-1)(2k-\alpha -1)d_k  \leq (k-1)^2 \theta_k  - k^2\theta_{k+1} + (2k-1)\theta_{k} \\
+ \frac{\beta}{\sqrt{s}}k\|x_{k} - x_{k-1} \|g_{k} + 2\beta^2   k^2 g_{k}^2. \label{basic10}
\end{multline}
Thus, for $k \geq k_0=(\alpha+1)/2 > 2$, we have
\begin{multline*}
k^2 d_{k+1} - (k-1)^2 d_k  \leq (k-1)^2 \theta_k  - k^2\theta_{k+1} + 2k\theta_{k}
+ \frac{\beta}{\sqrt{s}}k\|x_{k} - x_{k-1} \|g_{k} + 2\beta^2   k^2 g_{k}^2.
\end{multline*}
Summing from $k=k_0$ to $K > k_0$, we obtain
\begin{multline*}
K^2 d_{K+1}  \leq (k_0-1)^2 d_{k_0} + (k_0-1)^2 \theta_{k_0} + 2\sum_{k=k_0}^K k\theta_{k} + 2\beta^2  \sum_{k=k_0}^K  k^2 g_{k}^2 \\ 
+ \frac{\beta}{\sqrt{s}}\sum_{k=k_0}^K k\|x_{k} - x_{k-1} \|g_{k} .
\end{multline*}
Thanks to the estimates obtained in items \ref{pr.decay_E_k:claim2} and \ref{pr.decay_E_k:claim3} of Theorem~\ref{pr.decay_E_k}, we have respectively
\begin{equation}\label{g_k_square}
\sum_k k\theta_k<+\infty \qandq \sum_k  k^2 g_k^2 < +\infty .
\end{equation} 
In turn, there exists a constant $C > 0$ such that 
\begin{eqnarray*}
(K  \|x_{K+1} - x_{K} \|)^2  \leq C + C\sum_{k=1}^K  g_k \bpa{k\|  x_{k}  - x_{k-1}) \|}.
\end{eqnarray*}
Observe that
\begin{equation}\label{sum_g_k}
\sum_{k}  g_{k} <+\infty .
\end{equation}
This follows from {Cauchy-Schwarz} inequality, square-summability of the sequence $\seq{\frac{1}{k}}$ and the summability claim Theorem~\ref{pr.decay_E_k}\ref{pr.decay_E_k:claim2}. 
Thus, applying the Gronwall's lemma, we obtain 
\begin{equation*}\label{est_velocity}
\sup_k  k\|  x_{k}  - x_{k-1} \|<+\infty.
\end{equation*}
as claimed in \ref{pr.decay_W_k:claim1}. To see that assertion \ref{pr.decay_W_k:claim2} also holds, 
return to \eqref{basic10}, but this time without discarding the term $(\alpha-1)(2k-\alpha -1)d_k$ which is non-negative for $k \geq k_0$ since $\alpha > 1$. Summing \eqref{basic10}, using the telescopic form, the summability properties \eqref{g_k_square} and that of $\seq{g_k}$ as well as claim \ref{pr.decay_W_k:claim1}, we conclude.
\end{proof}

%%%%%%%%%%%%%%%%%%%%%%%%%%%%%
\subsection{Convergence of the iterates}
We are now ready to prove convergence of the iterates of algorithm \IGAHD.
\begin{theorem}\label{conv-iterates}
Let  $f: \cH \to \R$ be a convex function whose gradient is $L$-Lipschitz continuous. Let $\seq{x_k}$ be a sequence generated by  {\IGAHD}, where $\alpha > 3$, $0 < \beta <  2 \sqrt{s}$ and $sL \leq 1$. Then  $\seq{x_k}$ converges weakly, and its limit belongs to $S$.
\end{theorem}
\begin{proof}
Observe first that the sequence $\seq{x_k}$ is bounded in $\cH$. This is a direct consequence of Theorem~\ref{pr.decay_E_k} and Theorem~\ref{pr.decay_W_k}. Indeed, the convergence of the sequence  $\seq{E_k}$ from Theorem~\ref{pr.decay_E_k} implies that the sequence $\seq{v_k}$ in \eqref{Lyap-function-grad2} is bounded in $\cH$. From Theorem~\ref{pr.decay_E_k}\ref{pr.decay_E_k:claim3}, we know that $t_k \nabla f(x_k) \to 0$ as $k \to \infty$. On the other hand, we infer from Theorem~\ref{pr.decay_W_k}\ref{pr.decay_W_k:claim1} that $\seq{t_k(x_k-x_{k-1})}$ is bounded in $\cH$. It then follows from the definition of $v_k$ that the sequence $\seq{x_k}$ is bounded.

The rest of the proof now relies on Opial's Lemma~\ref{le.Opial} with $S=\argmin f$. By Theorem~\ref{pr.decay_E_k}\ref{pr.decay_E_k:claim1}, we have $f(x_k) \to \min_{\cH} f$. The weak lower semicontinuity of $f$ then gives item \ref{le.Opial:cond1} of Opial's Lemma. Thus, the only point to verify is that $\lim_{k\to\infty} \|x_k-x^\star\|$ exists for any $x^\star \in S$. Denote for short the anchor sequence $h_k \eqdef \frac{1}{2} \|x_k  - x^\star\|^2$.
Inspired by the continuous case, the idea of the proof consists in establishing a discrete second-order differential inequality satisfied by $\seq{h_k}$. We use the three-point identity
\[
\frac{1}{2} \|a-b\|^2 + \frac{1}{2} \|a-c\|^2 =
\frac{1}{2} \|b -c\|^2  + \left\dotp{a-b}{a-c \right},
\]
which holds for any  $a,b,c \in \cH$. Applying this identity at $b= x^\star$, $a= x_{k+1}$, $c=x_k$, we obtain
%\[
%\frac{1}{2} \|x_{k+1}-x^\star  \|^2 + \frac{1}{2} \|x_{k+1}-x_k  \|^2 =
%\frac{1}{2} \|  x_k - x^\star\|^2  + \left\dotp{x_{k+1}-x^\star}{x_{k+1}-x_k\right},
%\] 
%which is equivalent to
\begin{equation} \label{E:hk_k+1}
h_k - h_{k+1} = \frac{1}{2} \|x_{k+1}-x_k  \|^2  + \left\dotp{x_{k+1}-x^\star}{x_k - x_{k+1}\right} .
\end{equation} 
By  definition of $y_k$, we have 
\[
x_k - x_{k+1}= y_k - x_{k+1} - \alpha_k (x_k - x_{k-1}) + \beta \sqrt{s}   \left( \nabla f (x_k)  - \nabla f (x_{k-1}) \right) + \frac{\beta \sqrt{s}}{k}\nabla f( x_{k-1}).
\]
Therefore,
\begin{multline*}
h_k - h_{k+1} = \frac{1}{2} \|x_{k+1}-x_k  \|^2  + \left\dotp{x_{k+1}-x^\star}{y_k - x_{k+1}\right} -\alpha_k \left\dotp{x_{k+1}-x^\star}{x_k - x_{k-1}\right} \\
+ \left\dotp{x_{k+1}-x^\star}{\beta \sqrt{s} \pa{\nabla f (x_k)  - \nabla f (x_{k-1})} + \frac{\beta \sqrt{s}}{k}\nabla f( x_{k-1})\right} .
\end{multline*}
Since $\nabla f (x^\star)=0$ and $y_k - x_{k+1} = s \nabla f (y_k)$, we deduce that
\begin{multline}
h_{k+1}- h_k  +\frac{1}{2} \|x_{k+1}-x_k  \|^2   \\
+ s \dotp{ \nabla f (y_k) -\nabla f (x^\star)}{x_{k+1} -  x^\star}-\alpha_k \dotp{x_{k+1}-x^\star}{x_k - x_{k-1}}  \\
+ \left\dotp{x_{k+1}-x^\star}{\beta \sqrt{s}  \pa{\nabla f (x_k)  - \nabla f (x_{k-1})} + \frac{\beta \sqrt{s}}{k}\nabla f( x_{k-1})\right}  = 0 .\label{algo-conv16}
\end{multline}
On the other hand, the cocoercivity of $\nabla f$ and the hypothesis $sL\leq 1$ give
\begin{align}\label{algo-conv17}
&\dotp{ \nabla f (y_k) -\nabla f (x^\star)}{x_{k+1} -  x^\star} \nonumber \\
&=\dotp{ \nabla f (y_k) -\nabla f (x^\star)}{y_k  -x^\star} +\dotp{ \nabla f (y_k) -\nabla f(x^\star)}{x_{k+1} -  y_k} \nonumber\\
&\geq \frac{1}{L} \|\nabla f (y_k) -\nabla f (x^\star) \| ^2 +  \dotp{\nabla f(y_k) -\nabla f (x^\star)}{x_{k+1} -  y_k} \nonumber \\
&\geq s \|  \nabla f (y_k)  \| ^2 - s \|  \nabla f (y_k)\|^2 =0.
\end{align}
Injecting \eqref{algo-conv17} into \eqref{algo-conv16}, we get that
\begin{multline}
h_{k+1}- h_k  +\frac{1}{2} \|x_{k+1}-x_k  \|^2- \alpha_k \dotp{x_{k+1}-x^\star}{x_k - x_{k-1}}  \\
+ \left\dotp{x_{k+1}-x^\star}{\beta \sqrt{s}  \pa{\nabla f (x_k)  - \nabla f (x_{k-1})} + \frac{\beta \sqrt{s}}{k}\nabla f( x_{k-1})\right}  \leq 0 .\label{algo-conv18}
\end{multline}
By replacing $k$ by $k-1$ in \eqref{E:hk_k+1}, we obtain
\begin{equation}\label{algo-paral-3}
h_{k-1}-h_{k}  = \frac{1}{2} \|x_{k}-x_{k -1}  \|^2 - \dotp{x_{k}-x^\star}{x_{k}-x_{k -1}} .
\end{equation} 
By combining \eqref{algo-conv18} with \eqref{algo-paral-3}, we deduce that
\begin{multline}
h_{k+1}- h_k  - \alpha_k \pa{h_{k} - h_{k-1}} \leq 
 -\frac{1}{2} \|x_{k+1}-x_k  \|^2   \\
+\alpha_k \pa{\frac{1}{2} \|x_k -x_{k -1}  \|^2 + \dotp{x_k - x_{k-1}}{x_{k+1}-x_k}} \\
-\left\dotp{x_{k+1}-x^\star}{\beta \sqrt{s}\pa{\nabla f (x_k)  - \nabla f (x_{k-1})} + \frac{\beta \sqrt{s}}{k}\nabla f( x_{k-1})\right} .\label{algo-conv19}
\end{multline}
To simplify the exposition, let us use the generic terminology $e_k$ for the positive real terms which  are negligible with respect to the convergence property, and $C$ for the constants (independent of $k$). According to \cite{ACPR}, these are the terms that satisfy the summability property
\[
\sum_{k \in \N} k e_k < +\infty.
\]
This is the case of the term 
$|\dotp{x_{k+1}-x^\star}{\frac{\beta \sqrt{s}}{k}\nabla f( x_{k-1})}|$, since $\seq{x_k}$ is bounded as argued above, and 
$
k\|\frac{1}{k}\nabla f( x_{k-1})\| = \frac{1}{k}  \pa{k\|\nabla f( x_{k-1})\|}$, which is summable as the product of the two square summable sequences $\seq{\frac{1}{k}}$ and $\seq{k\nabla f( x_{k-1})}$ (the latter follows from Theorem~\ref{pr.decay_E_k}\ref{pr.decay_E_k:claim3}). 
Consequently, \eqref{algo-conv19} becomes
\begin{multline}
h_{k+1}- h_k  - \alpha_k \pa{h_{k} - h_{k-1}} \leq 
 -\frac{1}{2} \|x_{k+1}-x_k  \|^2   \\
+\alpha_k \pa{\frac{1}{2} \|x_k -x_{k -1}  \|^2 + \dotp{x_k - x_{k-1}}{x_{k+1}-x_k}} \\
-\beta \sqrt{s} \dotp{x_{k+1}-x^\star}{\nabla f (x_k)  - \nabla f (x_{k-1})} + C e_k. \label{algo-conv20}
\end{multline} 
By contrast, the term $\nabla f (x_k)  - \nabla f (x_{k-1}) $ is not negligible. Therefore, it will be treated as the difference of two consecutive terms, and will then handled easily. To see this, let us write 
\begin{align*}
&\dotp{x_{k+1}-x^\star}{\nabla f (x_k) - \nabla f (x_{k-1})} = \dotp{x_{k+1}-x^\star}{\nabla f (x_k)} - \dotp{x_{k+1}-x^\star}{\nabla f (x_{k-1})} \\
&= \dotp{x_{k}-x^\star}{\nabla f (x_k)} - \dotp{x_{k-1}-x^\star}{\nabla f (x_{k-1})}  \\
&\quad + \dotp{x_{k+1}-x_{k}}{\nabla f (x_k)}  
-\dotp{x_{k+1}-x_{k-1}}{\nabla f (x_{k-1})}  \\
&= \dotp{x_{k}-x^\star}{\nabla f (x_k)} - \dotp{x_{k-1}-x^\star}{\nabla f (x_{k-1})} + e_k.
\end{align*}
We have used that $\sum_{k \in \N} k \| x_{k+1}-x_{k}\| \|    \nabla f (x_k) \|<+\infty$, a consequence of {Cauchy-Schwarz} inequality and $\sum_{k \in \N} k \| x_{k+1}-x_{k}\|^2 <+\infty$ (Theorem~\ref{pr.decay_W_k}\ref{pr.decay_W_k:claim2}) and  $\sum_{k \in \N} k \|  \nabla f (x_k) \|^2 <+\infty$ (Theorem~\ref{pr.decay_E_k}\ref{pr.decay_E_k:claim3}). The same reasoning holds for $\sum_{k \in \N} k\| x_{k+1}-x_{k-1}\| \| \nabla f (x_{k-1}) \|<+\infty$.
Therefore
\begin{multline*}
h_{k+1}- h_k  - \alpha_k \pa{h_{k} - h_{k-1}} 
+\beta \sqrt{s}\dotp{x_{k}-x^\star}{\nabla f (x_k)} - \beta \sqrt{s}\dotp{x_{k-1}-x^\star}{\nabla f (x_{k-1})} \\
\leq -\frac{1}{2} \|x_{k+1}-x_k  \|^2  
 +\alpha_k \pa{\frac{1}{2} \|x_k -x_{k -1}  \|^2 + \dotp{x_k - x_{k-1}}{x_{k+1}-x_k}} +C  e_k. 
\end{multline*} 
Using again the estimation on the velocities (Theorem~\ref{pr.decay_W_k}\ref{pr.decay_W_k:claim2}), the second term in the right hand side of the last inequality is again negligible, which gives
\begin{equation}\label{algo-conv21}
 h_{k+1}- h_k  - \alpha_k \left(  h_{k} -h_{k-1}\right)
+\beta \sqrt{s}\left\dotp{x_{k}-x^\star}{\nabla f (x_k)  \right} -\beta \sqrt{s} \left\dotp{x_{k-1}-x^\star}{\nabla f (x_{k-1}) \right}  \leq C e_k.
\end{equation}
Set
$
\theta_k \eqdef  h_{k} -h_{k-1} + \beta \sqrt{s} \left\dotp{x_{k-1}-x^\star}{\nabla f (x_{k-1}) \right}.
$
From \eqref{algo-conv21} we infer that
\[
\theta_{k+1} -\alpha_k  \theta_k \leq \frac{\alpha}{k}\beta \sqrt{s} \left\dotp{x_{k-1}-x^\star}{\nabla f (x_{k-1}) \right} + C  e_k
\]
By the same argument as above we deduce that
\[
\theta_{k+1} - \alpha_k  \theta_k \leq  C  e_k .
\]
The proof now follows the line of \cite{ACPR}. Taking the positive part we arrive at
\[
\brac{\theta_{k+1}}_+ \leq \alpha_k \brac{\theta_k}_+  + C  e_k .
\] 
Applying Lemma~\ref{diff-ineq-disc} with $a_k = \brac{\theta_k}_+$, we obtain
\begin{center}
$\sum_{k \in \N} \brac{h_{k} -h_{k-1}+ \beta \sqrt{s} \dotp{x_{k-1}-x^\star}{\nabla f (x_{k-1})}}_
+ < +\infty.$
\end{center}
Since $\seq{\| \nabla f (x_{k-1}) \|}$ is summable (see \eqref{sum_g_k}), we deduce that $\sum_{k }\brac{h_{k} -h_{k-1}}_+ < +\infty$, which implies that the limit of the sequence $\seq{h_k}$ exists.
Condition~\ref{le.Opial:cond2} of Lemma~\ref{le.Opial} is then verified which concludes the proof.
\end{proof}
%
%%%%%%%%%%%%%%%%%%%%%%%%%
\subsection{From $\cO\pa{\frac{1}{k^2}}$ to ${\rm o}\pa{\frac{1}{k^2}}$ rates}
We will now establish an even faster asymptotic rate of convergence of the values and velocities. 
\begin{theorem}\label{small_o}
Let  $f: \cH \to \R$ be a convex function whose gradient is $L$-Lipschitz continuous. Let $\seq{x_k}$ be a sequence generated by  {\IGAHD}, where $\alpha > 3$, $0 < \beta <  2 \sqrt{s}$ and $sL \leq 1$. Then
\begin{enumerate}[label={\rm (\roman*)}]
\item $f(x_k)-\min_{\cH} f = {\rm o}\pa{\dfrac{1}{k^2}}$ as $k\to +\infty$; \label{small_o:claim1}
\item $\norm{x_{k}  - x_{k-1}} = {\rm o}\pa{\dfrac{1}{k^2}}$ as $k\to +\infty$. \label{small_o:claim2}
\end{enumerate}
\end{theorem}
\begin{proof}
Let us embark from \eqref{basic10} and recall the notations there. Let us define $W_k \eqdef (k-1)^2 d_k + (k-1)^2 \theta_k$. We then have for $k \geq (\alpha+1)/2$,
\[
W_{k+1} \leq W_k +e_k, \text{ where } e_k = 2k\theta_k + \frac{\beta}{\sqrt{s}}k\norm{x_k-x_{k-1}}g_k + 2\beta^2k^2g_k^2.
\] 
Thanks to \eqref{g_k_square}, \eqref{sum_g_k} and Theorem~\ref{pr.decay_W_k}\ref{pr.decay_W_k:claim1}, we have $\sum_{k \in \N} e_k < +\infty$.

%that we recall below, 
%\begin{eqnarray}
% &&k^2 d_{k+1} -(k-1)^2 d_k  +(\alpha-1)(2k-\alpha -1)d_k  \leq k^2 \theta_k  - (k+1)^2\theta_{k+1} + (2k+1)\theta_{k+1} \nonumber\\
%&+&    \frac{\beta}{\sqrt{s}} k\sqrt{d_{k}}g_{k} + 2\beta^2   k^2 g_{k}^2, \label{basic10_b}
%\end{eqnarray}
%where $\theta_k \eqdef f (x_k)  -f(x^\star)$ and $d_k\eqdef \frac{1}{2s} \| x_{k} -x_{k-1} \|^2 $.
It follows that the limit of $W_k$ exists as $k \to +\infty$. This limit $\ell$ is necessarily equal to zero, otherwise, for $k$ sufficiently large, $W_k \geq \frac{\ell}{2}$, which gives
\[
(k-1) d_k + (k-1) \theta_k \geq \frac{l}{2k}.
\]
This gives a clear contradiction with the summability of the left hand side of the above inequality, as given by Theorems~\ref{pr.decay_E_k}\ref{pr.decay_E_k:claim2} and \ref{pr.decay_W_k}\ref{pr.decay_W_k:claim2}. This concludes the proof.
\end{proof}

%%%%%%%%%%%%%%%%%%%%%%%%%%%%%%%%%%%%%%%%%%%%%%%%%%%%%%%%%%%%%%%%%%%%%%%%%%%%%%%%%%%
\section{Application}\label{sec:numerics}
%%%%%%%%%%%%%%%%%%%%%%%%%%%%%%%%%%%%%%%%%%%%%%%%%%%%%%%%%%%%%%%%%%%%%%%%%%%%%%%%%%%
%
To illustrate our results, let us consider the regularized least-squares problem \eqref{eq:minP} 
\begin{equation}\label{eq:minP}\tag{RLS}
\min_{x \in \R^n} \left\{ f(x) \eqdef \frac{1}{2}\norm{b-Ax}^2 + g(x) \right\} ,
\end{equation}
on $\cH=\R^n$, where $A$ is a linear operator from $\R^n$ to $\R^m$, $m \leq n$, $g: \R^n \to \Rb$ is a proper lower semicontinuous and convex function which acts as a regularizer.~\eqref{eq:minP} occurs in a variety of fields ranging from inverse problems in signal/image processing, to machine learning and statistics. Typical examples for $g$ include the $\ell_1$ norm (Lasso), the $\ell_1-\ell_2$ norm (group Lasso), the total variation, or the nuclear norm. We assume that the set of minimizers of~\eqref{eq:minP} is non-empty.\\
Following \cite{ACFR}, the key idea is to work with an  appropriate metric. For a symmetric positive definite matrix $S \in \R^{n \times n}$, denote by $\norm{\cdot}_S$  the norm which is associated with the  scalar product  $\dotp{S\cdot}{\cdot}$. 
%When $S=I$, then we simply use the shorthand notation for the Euclidean scalar product $\dotp{\cdot}{\cdot}$ and norm $\norm{\cdot}$.
For a proper convex lsc function $h$,  denote $h_S$ and $\prox_h^S$ its Moreau envelope and proximal mapping in the metric $S$, 
\[
h_S(x) = \min_{z \in \R^n} \frac{1}{2}\norm{z-x}_S^2 + h(z), \quad \prox^S_{h}(x) = \argmin_{z \in \R^n} \frac{1}{2}\norm{z-x}_S^2 + h(z) .
\] 
%Similarly, when $S=I$, we drop $S$ in the above notation.
Let $M=\lambda^{-1}I-A^*A$. With the proviso that $0 < \lambda\norm{A}^2 < 1$, $M$ is a symmetric positive definite matrix. It can be easily shown, see \cite[Appendix]{ACPR}, that 
\[
\prox_{f}^M (x) = \prox_{\lambda g}(x+\lambda A^*(b-x)) ,
\]
and that $f_M$ is a continuously differentiable convex function whose gradient in the metric $M$ is given by the formula
\[
\nabla f_M(x) = x - \prox_{f}^M (x) = x - \prox_{\lambda g}(x + \lambda A^*(b-Ax)) .
\]
Moreover, $\norm{\nabla f_M(x)-\nabla f_M(z)}_M \leq \norm{x-z}_M$, \ie $\nabla f_M$ is Lipschitz continuous in the metric $M$. In addition, a standard argument shows that 
$
\argmin(f) =  \argmin(f_M) .
$
We are then in position to solve \eqref{eq:minP} by simply applying \IGAHD to $f_M$.
We obtain the following algorithm: \\
\renewcommand{\algorithmcfname}{$\rm(IGAHD-RLS)$}
\begin{algorithm}[H]
\caption{}
Initialization: $x_0, x_1 \in \R^n$ given\;
\For{$k=1,\ldots $}{
$z_k= x_k-\prox_{\lambda g}(x_k+\lambda A^*(b-Ax_k))$ \;  
$y_k=   x_{k} + (1- \frac{\alpha}{k}) ( x_{k}  - x_{k-1}) - \beta \sqrt{s} (z_k - z_{k-1}) -\dfrac{\beta \sqrt{s}}{k} z_k$ \;
$x_{k+1}= (1-s)y_{k} + s\prox_{\lambda g}(y_k+\lambda A^*(b-Ay_k))$.
}
\end{algorithm}

We infer from Theorems~\ref{pr.decay_E_k}, \ref{pr.decay_W_k}, and \ref{small_o} the following convergence certificate of algorithm $\rm(IGAHD-RLS)$.
\begin{theorem}\label{Lasso-thm}
Consider algorithm {\rm (IGAHD-RLS)} applied with $\alpha > 3$, $0 < \lambda\norm{A}^2 < 1$, $0 < \beta <  2 \sqrt{s}$ and $s \leq 1$. Then for any sequence  $\seq{x_k}$ generated by the algorithm {\rm (IGAHD-RLS)}, the following properties hold:
\begin{enumerate}[label={\rm (\roman*)}]
\item $f(\prox_{f}^M (x_k))-\min_{\cH} f = {\rm o}(k^{-2})$ and $\sum_{k \in \N}  k^2 \| \nabla f (x_k) \|^2 < +\infty$.
\item The sequence $\seq{x_k}$ converges to a solution of \eqref{eq:minP}.
\end{enumerate}
\end{theorem}
%
%The above result gives a general certificate of convergence of iterates for Lasso problems, see \cite{Tibshirani} for a discussion of uniqueness for this type of problem.

\IGAHD and FISTA (\ie \IGAHD with $\beta=0$) were applied to $f_M$ with two instances of $g$: $\ell_1$ norm (for sparse vector recovery) and the nuclear norm (for low-rank matrices recovery). The results are depicted in Figure~\ref{fig:rls}. We display both the evolution of the objective values and the distance to the set of minimizers, for different values of the damping parameter $\alpha > 3$. \IGAHD exhibits less oscillations than FISTA, and eventually converges faster both on the values and iterates.

\begin{figure}
  \includegraphics[width=\textwidth]{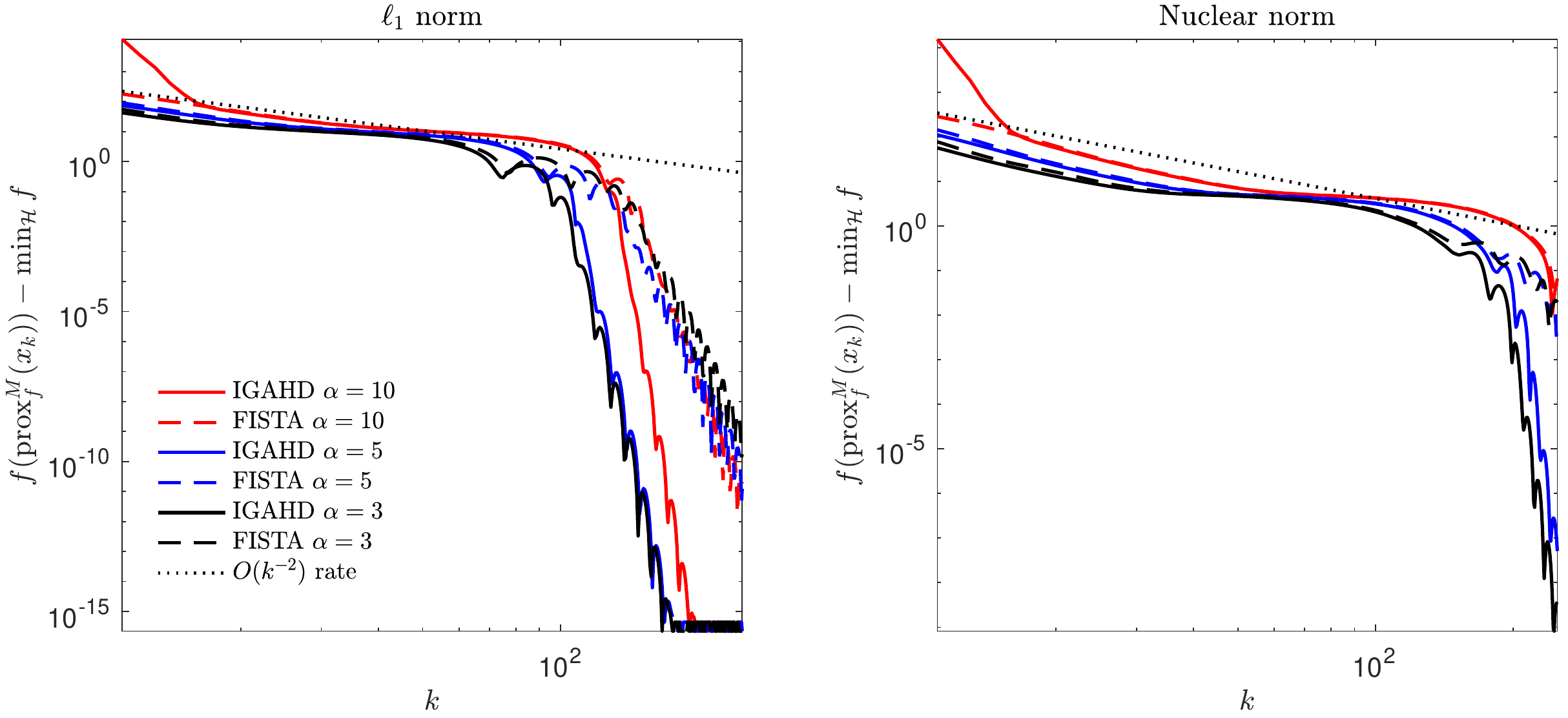}
  \includegraphics[width=\textwidth]{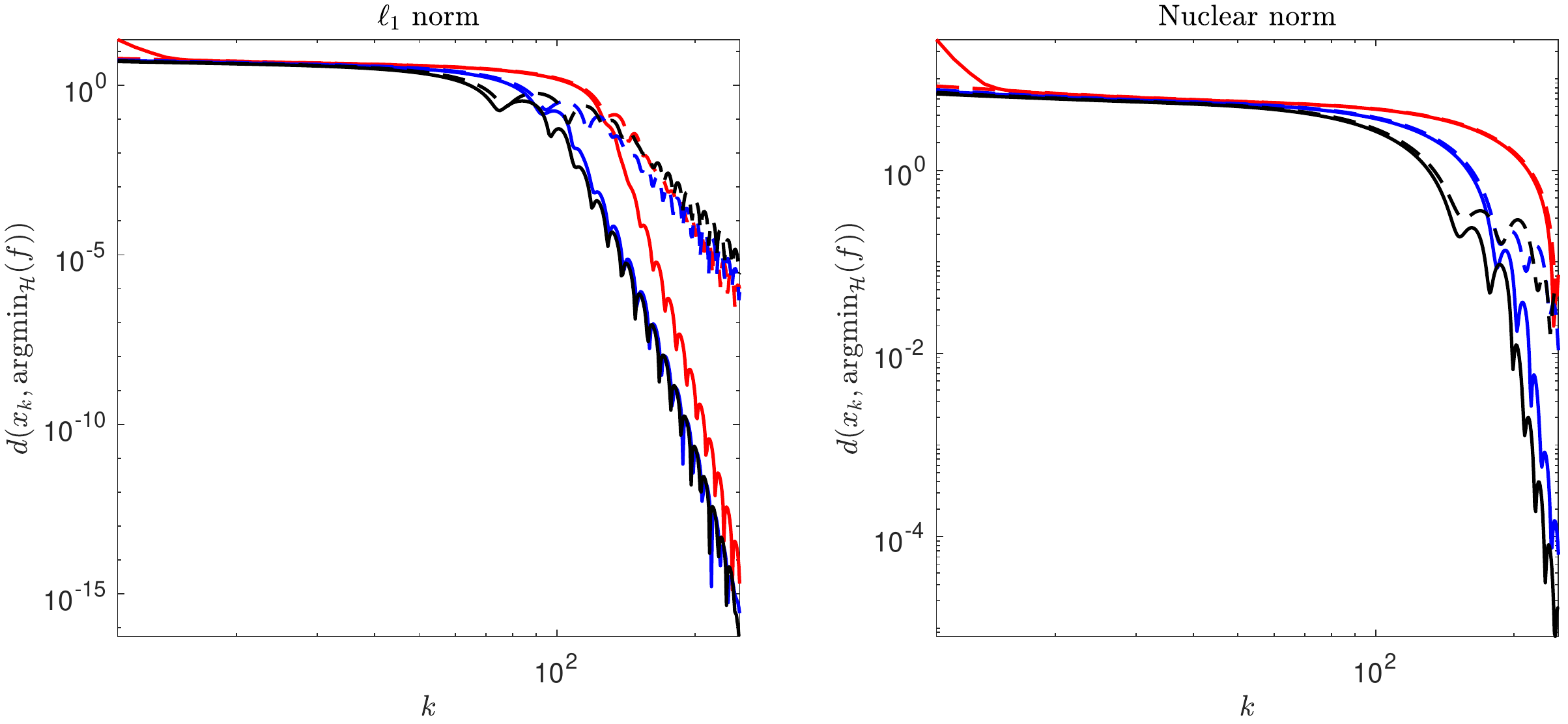}
\caption{Evolution of $f(\prox^M_f(x_k))-\min_{\R^n} f$ and the distance of $x_k$ to $\argmin_{\R^n} f$, where $x_k$ is the iterate of either \IGAHD or FISTA, when solving \eqref{eq:minP} with different regularizers $g$.}
\label{fig:rls}
\end{figure}

%%%%%%%%%%%%%%%%%%%%%%%%%%%%%%%%%%%%%%%%%%%%%%%%%%%%%%%%%%%%%%%%%%%%%%%%%%
\appendix \section{Auxiliary Lemmas}
%%%%%%%%%%%%%%%%%%%%%%%%%%%%%%%%%%%%%%%%%%%%%%%%%%%%%%%%%%%%%%%%%%%%%%%%%%
%
\begin{lemma}[{\cite[Opial's Lemma]{Opial}}]\label{le.Opial} Let $S$ be a nonempty subset of $\cH$,
and $\seq{x_k}$ a sequence of elements of $\cH$. Assume that 
\begin{enumerate}[label={\rm (\roman*)}]
\item every sequential weak cluster point of $(x_k)$, as $k\to+\infty$, belongs to $S$; \label{le.Opial:cond1}
\item for every $z\in S$, $\lim_{k\to+\infty}\|x_k-z\|$ exists. \label{le.Opial:cond2}
\end{enumerate}
Then $\seq{x_k}$ converges weakly as $k\to+\infty$ to a point in $S$. 
\end{lemma}
\begin{lemma}\label{le.conv-seq} Let $\seq{v_k}$ be a sequence in $\cH$. Assume $\alpha-1>0$ and that the sequence $\seq{a_k}$ defined by
\begin{eqnarray}\label{basic-formula-l-1}
a_k \eqdef q_k+ \frac{k}{\alpha -1}\pa{q_k-q_{k-1}}.
\end{eqnarray}
strongly converges to some limit. Then $\seq{q_k}$ strongly converges to the same limit as $k\to+\infty$. 
\end{lemma}
\begin{lemma}\label{diff-ineq-disc} 
Given $\alpha\ge 3$, let $\seq{a_k}$, $\seq{\omega_k}$ be two sequences of non-negative numbers such that
\[
a_{k+1} \leq \pa{1-\frac{\alpha}{k}} a_k + \omega_k
\]
for all $k\geq 1$. If $\sum_{k\in\mathbb N} k\omega_k <+\infty$, then $\sum_{k \in \mathbb N} a_k < +\infty$.
\end{lemma}

\end{document}